\documentclass[11pt]{amsart}
\usepackage[utf8]{inputenc}
\usepackage{listings}
\usepackage{comment}
\usepackage{lmodern}
\usepackage{graphicx}
\usepackage{custom-template,shortcuts}

\title[The geometric step via resultants]{Polylogarithmic motivic Chabauty--Kim for $\mathbb{P}^1 \setminus \{ 0,1,\infty \}$: \\ the geometric step via resultants}
\author[Jarossay]{David Jarossay}
\address{David Jarossay, De Vinci Higher Education, De Vinci Research Center, Paris, France\vspace*{-3pt}}
\email{david.jarossay@devinci.fr}
\author[Lilienfeldt]{David T.-B. G. Lilienfeldt}
\address{David T.-B. G. Lilienfeldt, Leiden University, Mathematical Institute, The Netherlands \newline
Current address: \'Ecole Polytechnique, Centre de Math\'ematiques Laurent Schwartz, France \vspace*{-3pt}}
\email{david.lilienfeldt@polytechnique.edu}
\author[Saettone]{Francesco M. Saettone}
\address{Francesco M. Saettone, Weizmann Institute of Science, Department of Mathematics, Rehovot, Israel\vspace*{-3pt}}
\email{francesco.saettone@weizmann.ac.il}
\author[Weiss]{Ariel Weiss}
\address{Ariel Weiss, The Ohio State University, Department of Mathematics, USA \newline
Current address: Trinity College, Department of Mathematics, Hartford, CT 06106, USA\vspace*{-3pt}}
\email{ariel.weiss@trincoll.edu}
\author[Zehavi]{Sa'ar Zehavi}
\address{Sa'ar Zehavi, Ben-Gurion University of the Negev, Be'er Sheva, Israel \newline Current address: Weizmann Institute of Science, Department of Mathematics, Rehovot, Israel \vspace*{-3pt}}
\email{saar.zehavi@weizmann.ac.il}
\date{\today}

\begin{document}

\begin{abstract}
Given a finite set $S$ of primes, we propose a method to construct polylogarithmic motivic Chabauty--Kim functions for $\mathbb{P}^1 \setminus \{ 0,1,\infty \}$ using resultants. For a prime $p\not\in S$, the vanishing loci of the images of such functions under the $p$-adic period map contain the solutions of the $S$-unit equation. In the case $\vert S\vert=2$, we explicitly construct a non-trivial motivic Chabauty--Kim function in depth 6 of degree 18, and prove that there do not exist any other Chabauty--Kim functions with smaller depth and degree. The method, inspired by work of Dan-Cohen and the first author, enhances the geometric step algorithm developed by Corwin and Dan-Cohen, providing a more efficient approach.
\end{abstract}

\maketitle

\tableofcontents

\section{Introduction}

Let $S$ be a finite set of primes, and let $X = \P^1\setminus\{0,1,\infty\}$ denote the thrice punctured projective line over the ring of $S$-integers $\Z_S$ of $\Q$. By Siegel's theorem, the set $X(\Z_S)$ is finite. Equivalently, the $S$-unit equation $a + b = 1$ with $a, b\in \Z_S\t$ has finitely many solutions. In an effort to determine $X(\Z_S)$ explicitly, Kim \cite{kim} initiated the non-abelian Chabauty program. This $p$-adic method ($p$ a prime not in $S$) produces a nested sequence of sets
\[
X(\Z_p)\supseteq X(\Z_p)^{\mathrm{Kim}}_{S,1} \supseteq X(\Z_p)^{\mathrm{Kim}}_{S,2} \supseteq \cdots \supseteq X(\Z_p)^{\mathrm{Kim}}_{S,d} \supseteq \cdots \supseteq X(\Z_S)
\]
that become finite for $d\gg 0$, and Kim conjectured that $X(\Z_S)=X(\Z_p)^{\mathrm{Kim}}_{S,d}$ for $d$ sufficiently large. In particular, computing these sets gives a way to explicitly determine $X(\Z_S)$. 

The set $X(\Z_p)^{\mathrm{Kim}}_{S,d}$, called the \emph{Chabauty--Kim locus} in depth $d$, is defined as the vanishing locus of certain polynomials in $p$-adic multiple polylogarithms. Recently, in \cites{CDC1, CDC2}, Corwin and Dan-Cohen proposed a motivic version of Kim's method for the curve $X$ that only requires the use of $p$-adic single polylogarithms, rather than multiple polylogarithms. Their method produces a nested sequence of sets
\[
X(\Z_p)\supseteq X(\Z_p)^{\mathrm{PL}}_{S,1} \supseteq X(\Z_p)^{\mathrm{PL}}_{S,2} \supseteq \cdots \supseteq X(\Z_p)^{\mathrm{PL}}_{S,d} \supseteq \cdots \supseteq X(\Z_S),
\] 
satisfying $X(\Z_p)^{\mathrm{Kim}}_{S,d} \subseteq (X(\Z_p)^{\mathrm{PL}}_{S,d})^{S_3}$, where the latter denotes the largest $S_3$-stable subset of $X(\Z_p)^{\mathrm{PL}}_{S,d}$. For $d$ sufficiently large, Corwin--Dan-Cohen \cite{CDC1}*{Conj.~2.32} conjectured the equality $X(\Z_S)=(X(\Z_p)^{\mathrm{PL}}_{S,d})^{S_3}$, which implies Kim's conjecture. 

In practice, it is difficult to compute the sets $X(\Z_p)^{\mathrm{Kim}}_{S,d}$ or $X(\Z_p)^{\mathrm{PL}}_{S,d}$: so far, they have only been computed when $\vert S\vert=0$ \cite{BDCKW} and $\vert S\vert=1$ with $d\leq 4$ \cites{CDC1,DCW1}. 

The main advantage of the motivic approach of \cites{CDC1, CDC2} is the separation of the computation of $X(\Z_p)^{\mathrm{PL}}_{S,d}$ into two steps:
\begin{itemize}[leftmargin=*]
\item {\it Geometric step}: find non-trivial polylogarithmic motivic Chabauty--Kim functions in depth $\leq d$.
These functions (\Cref{def:I}) are elements of a weighted polynomial ring $\O(U_S)[\log^{\fu}, \Li^{\fu}_1,\ldots,\Li^{\fu}_d]$, which vanish under a certain cocyle map $\theta^\#$ (\Cref{theta}). Here:
\begin{itemize}
    \item the variables are the unipotent de Rham polylogarithms $\log^{\fu}$ and $\Li_n^{\fu}$ for $n\geq 1$ (see \Cref{s:dRfund});
    \item the coefficient ring $\O(U_S)$ is the coordinate ring of the pro-unipotent mixed Tate fundamental group over $\Z_S$ (see \Cref{s:proT});
    \item the variable $\log^{\fu}$ is assigned degree $1$ and each $\Li^{\fu}_n$ is assigned degree $n$. 
\end{itemize}
This step depends on the cardinality of $S$, but not on the specific primes contained in $S$, nor on the prime $p$.
\item {\it Arithmetic step}: determine the images of the functions found in the geometric step under the $p$-adic period map, in order to obtain $p$-adic Chabauty--Kim functions. \footnote{Can be checked to be non-trivial in practice and are guaranteed to be non-trivial if one assumes the $p$-adic Period Conjecture \cite[Conj. 2.25]{CDC1}.} This step depends on the prime $p$ and the specific primes contained in $S$. 
\end{itemize}
In this paper, we focus on the geometric step and propose a method for solving it for any set $S$ based on the use of resultants. In the case $\vert S\vert=2$, we carry out this resultant method explicitly:
\begin{theorem}\label{thm:intro}
    Assume that $\vert S\vert=2$. 
    \begin{enumerate}
    \item There do not exist any motivic Chabauty--Kim functions of degree less than $18$. 
    \item The lowest depth at which a motivic Chabauty--Kim function exists is $6$. 
    \item There exists a non-trivial motivic Chabauty--Kim function $F^{\vert 2\vert}_{6,18}$ of depth $6$ and degree $18$ (see \Cref{def:F618}). Moreover, any other motivic Chabauty--Kim function of depth less than $19$ and degree $18$ is a multiple of $F^{\vert 2\vert}_{6,18}$ by an element of $\oh(U_S)$. 
    \end{enumerate}
\end{theorem}

This result constitutes the first progress towards explicit Chabauty--Kim in depth $> 4$ for $\vert S\vert\geq 2$. 

\begin{remark}
The cocycle map $\theta^\sharp$ can be viewed as a morphism of graded vector spaces, and as the depth $d$ and the degree $v$ grow, the dimension of the relevant graded piece of the domain grows faster than that of the corresponding graded piece of the codomain (\Cref{prop:dpl-dphi}); we tabulate these dimensions in \Cref{tab:s=1}. In particular, when $d$ and $v$ are large enough, the existence of a non-trivial motivic Chabauty--Kim function is guaranteed by a simple dimension argument.
However, the non-trivial motivic Chabauty--Kim function $F^{\vert 2\vert}_{6,18}$ of \Cref{thm:intro} does not fit into this regime: the depth $6$, degree $18$ part of the domain has dimension $996$, while the corresponding part of the codomain has dimension $4183$. 
Hence, the mere existence of $F^{\vert 2\vert}_{6,18}$ is surprising! 
At least when $|S|=2$, our construction can produce multiple distinct non-trivial motivic Chabauty--Kim functions whose existence cannot be explained by naive dimension estimates. It would be interesting to see what circumstances could give an a priori justification for their existence. These observations offer hints of new phenomena in Chabauty--Kim theory that are not visible when $\vert S\vert=1$ (see Section \ref{sec:s=1}).
\end{remark}

\begin{remark}
    The case $\vert S\vert=2$ and $d=5$ is the first instance for which the Goncharov subalgebra $\oh^G(U_S)$ (see \cite[\S 2.3.2]{CDC2}) is strictly contained in $\oh(U_S)$. While the coefficients of $F^{\vert 2\vert}_{6,18}$, which are rational linear combinations of the generators of the shuffle algebra $\oh(U_S)$, are guaranteed to lie in $\oh^G(U_S)$ by \cite[Prop.~2.3.5]{CDC2}, the individual terms of these linear combinations do not belong to $\oh^G(U_S)$ in general. According to Corwin--Dan-Cohen's integral strengthening of Goncharov's depth-$1$ conjecture \cite[Conj.~2.3.6]{CDC2}, the coefficients of $F^{\vert 2\vert}_{6,18}$ should therefore be expressible as rational linear combinations of {\it single} polylogarithmic values (possibly requiring enlarging $S$), even though the individual terms might require {\it multiple} polylogarithmic values in order to carry out the arithmetic step. In the future, it would be interesting to observe these phenomena in the arithmetic step for $F^{\vert 2\vert}_{6,18}$.
\end{remark}

\subsection{Polylogarithmic motivic Chabauty--Kim}

The method of \cite{CDC1} rests on the following commutative diagram of sets:
\begin{equation}\label{diagintro}
\begin{tikzcd}[bend angle = 80]
  X(\Z_S) \arrow[hook]{r} \arrow{d}{\kappa} & X(\Z_p) \arrow{d}{\kappa_p} \\
  \Sel_S^{\PL}(\Q) \arrow{r}{\loc} & \Pi^{\PL}(\Q_p).
\end{tikzcd}
\end{equation}

We refer to \Cref{s:diag} for a detailed description of this diagram. In this context:
\begin{itemize}
    \item $\Pi^{\PL}$ denotes the polylogarithmic pro-unipotent de Rham fundamental group of $X$ (see \Cref{s:dRfund});
    \item the Selmer scheme is defined as $\Sel_S^{\PL}=\Hom_{\G_m}(U_S, \Pi^{\PL})$, where $U_S$ denotes the pro-unipotent mixed Tate fundamental group over $\Z_S$ (see \Cref{s:proT});
    \item the map $\kappa$ is the (unipotent) motivic Kummer map, which is a map of sets;
    \item the map $\kappa_p$ is the local (unipotent) Kummer map, which is a $p$-adic locally analytic map;
    \item the map $\loc$ evaluates a homomorphism at the $p$-adic period map $\per_p\in U_S(\Q_p)$. The map $\loc$ is a map of $\Q_p$-schemes restricted to $\Q$-points.
\end{itemize}

The Selmer scheme comes equipped with a universal evaluation map of $U_S$-schemes 
\begin{equation}\label{intro:ev}
\ev \colon \Sel_S^{\PL} \times_{\Q}\, U_S \to \Pi^{\PL}\times_{\Q} U_S.
\end{equation}
Let $K$ denote the function field of $U_S$ and let $\ev_K$ be the pull-back of $\ev$ along $\spec(K)\to U_S$. Given an integer $d\geq 1$, define 
\[
\mathcal{I}^{\PL}_{S,d}:=\O(\Pi^{\PL}_{\geq -d}\times (U_S)_{\geq -d})\cap \ker(\ev_K^{\#}),
\]
where we used the notation of \cite[bottom p. 1869]{CDC1} for depth $d$ quotient pro-unipotent groups.
Via pull-back by the $p$-adic period map, any function $f\in \mathcal{I}^{\PL}_{S,d}$ induces a $p$-adic Chabauty--Kim function (see \eqref{pCK} for the precise definition)
\[
f\vert_{X(\Z_p)} \colon X(\Z_p) \xrightarrow{\kappa} \Pi^{\PL}(\Q_p) \xrightarrow{f_{\Q_p}} \Q_p.
\]
The {\it polylogarithmic motivic Chabauty--Kim locus} in depth $d$ is then defined as
\[X(\Z_p)^{\PL}_{S, d} := \set{z\in X(\Z_p) : f|
_{X(\Z_p)}(z) = 0,\ \ \forall f\in \mathcal{I}^{\PL}_{S,d}}.\]
The geometric step amounts to finding non-zero elements of $\mathcal{I}^{\PL}_{S,d}$.

\subsection{Methods}

In \cite[\S 3.3]{CDC1}, Corwin and Dan-Cohen describe an affine model for the Selmer scheme $\Sel^{\PL}_S$ in terms of non-canonical coordinates. In this model, the map \eqref{intro:ev} is identified with an explicit map (Definition \ref{theta})
\begin{equation}\label{intro:theta}
    \theta^{\#} \: \oh(U_S)[\log^{\fu}, \Li_1^{\fu}, \Li_2^{\fu}, \ldots] \to \oh(U_S)[\Phi]
\end{equation}
of graded algebras over the shuffle algebra $\O(U_S)$, whose kernel is exactly the space of motivic Chabauty--Kim functions. Restricting to a fixed depth $d$ and degree $v$, there is a corresponding $\O(U_S)_{\le d}$-module map (see Definition \ref{def:az-le-d} for the definition of $\O(U_S)_{\le d}$)
\begin{equation}\label{intro:thetav}
    \theta^{\#}_{d, v} \: \oh(U_S)_{\le d}[\log^{\fu}, \Li_1^{\fu}, \ldots\Li_d^{\fu}]_v \to \oh(U_S)_{\le d}[\Phi, d]_v.
\end{equation}
Completing the geometric step in depth $d$ and degree $v$ is equivalent to computing the kernel of this map.

When $|S| = 1$, Corwin and Dan-Cohen compute the kernel of $\theta^{\#}_{d, v}$ for small $d$ and $v$ by hand, with minimal difficulty. However, this case is deceptively simple: we show in \Cref{rem:infeasible} that for $|S| = 2$, even when one knows that $\ker(\theta^{\#}_{6, 18})$ should be non-trivial, computing its kernel directly amounts to computing the kernel of a $4183\times 996$ matrix over a polynomial algebra in $30$ variables, which is infeasible even for a supercomputer.

To prove \Cref{thm:intro}, we explicitly compute this kernel using a two-faceted approach. First, we develop an algorithm to prove that $\ker(\theta^{\#}_{d, v}) = 0$ whenever $d, v\le 17$. Second, via the method of resultants, we explicitly construct an element of $\ker(\theta^{\#}_{6, 18})$.

\subsubsection{An algorithm for bounding the number of motivic Chabauty--Kim functions}\label{s:introalgo}

The shuffle algebra $\oh(U_S)_{\leq d}$ can be identified with a polynomial ring $\Q[X_1, \ldots, X_n]$ with $n$ the number of Lyndon words of degree $\geq -d$ of a fixed ordered generating set $\Sigma$ for $\Lie(U_S)$. Thus, for fixed $d$ and $v$, the operator $\theta^{\#}_{d, v}$ can be viewed as a matrix over a polynomial ring. While computing its kernel directly is infeasible, observe that its kernel is non-zero only if it is non-zero for every possible integer evaluation of $(X_1, \ldots, X_n)$. 

We use this observation to formulate \Cref{algo}, which takes as input a tuple of positive integers $(s,d,v)$, where $s=\vert S\vert$, and outputs the dimension of the kernel of $\theta^{\#}_{d, v}$, evaluated at a random integer vector $x\in \Z^n$. The output is therefore an upper bound (and likely also a lower bound -- see Remark \ref{rem:Zar}) for the dimension of the space of polylogarithmic motivic Chabauty--Kim functions in depth $d$ and degree $v$ for $s$ primes (\Cref{prop:guestimate}).

We have implemented this algorithm in SageMath \cite{sagemath}; the code can be viewed at our  \href{https://github.com/Ariel-Z-Weiss/polylog-chabauty-kim}{GitHub repository}. The output of the algorithm for $(2, 17, 17)$ is $0$, proving \Cref{thm:intro}$(i)$ (\Cref{thm:no}). 

However, the output of the algorithm for $(2,17,18)$, and even $(2, 6, 18)$, is $1$, which indicates that it is extremely likely that a non-zero polylogarithmic motivic Chabauty--Kim function exists in depth $6$ and degree $18$. Granted the existence of such a function, the output proves \Cref{thm:intro}$(iii)$. In order to prove \Cref{thm:intro}, all that remains is therefore to prove the existence of such a function.

\subsubsection{Constructing motivic Chabauty--Kim functions via resultants}

For any $\vert S\vert$, we propose a resultant method that effectively constructs polylogarithmic motivic Chabauty--Kim functions (\Cref{s:multiple}). The method is inspired by the work \cite{DCJ} on the surface $M_{0,5}$ (the moduli space of genus $0$ curves with $5$ marked points). At this level of generality, we cannot guarantee the non-triviality of the resulting functions. However, when $\vert S\vert=2$, we construct a function of depth $6$ and degree $18$ and prove its non-triviality (\Cref{prop:F5}). We now briefly sketch the construction in the case $\vert S\vert =2$ (see \Cref{chapter:resultants} for the detailed construction). 

Constructing a polylogarithmic motivic Chabauty--Kim function requires finding an element in the kernel of the map \eqref{intro:theta}. As explained in \Cref{s:introalgo}, after restricting to a fixed depth $d$ and degree $v$, the map \eqref{intro:thetav} can be expressed as a matrix with coefficients in the shuffle algebra $\oh(U_S)_{\leq d}$. It is however computationally infeasible to compute its null space, even though we can implement the matrix in SageMath. In the case where $\vert S\vert=2$ and $d$ is even, the images of $\log^{\fu}, \Li_1^{\fu}, \ldots, \Li_d^{\fu}$ under \eqref{intro:theta} take values in a polynomial algebra with $4+(d-1)/2$ variables $x_1, y_1, x_2, y_2, z_3, z_5, \ldots, z_{d-1}$ and coefficients in $\oh(U_S)_{\leq d}$. Due to the particular shape of the map \eqref{intro:theta}, it is possible to express $z_k$ in terms of images under \eqref{intro:theta} and the variables $x_1, y_1, x_2, y_2,$ and $z_{k'}$ for $k'<k$. From the shape of $\theta^{\#}(\log^{\fu}), \theta^{\#}(\Li_1^{\fu})$, and $\theta^{\#}(\Li_2^{\fu})$, the variables $x_1, y_1, x_2, y_2$ can in turn be expressed in terms of images under \eqref{intro:ev} and the single variable $x_1$. From these observations, it is possible to derive a polynomial expression $P_d(X)\in \mathrm{Im}(\theta^{\#})[X]$ with the property that $P_d(x_1)=0$. The construction works whenever $d\geq 4$. Write $P_d(X)=\theta^{\#}(\nu_d(X))$ for some $\nu_d(X)\in (\oh(U_S)_{\leq d}[\log^{\fu}, \Li_1^{\fu}, \ldots, \Li_d^{\fu}])[X]$. Then 
\[
    \mathrm{Res}(\nu_4(X), \nu_6(X))\in \mathcal{I}^{\PL}_{S, 6}
\]
is a non-trivial Chabauty--Kim function, which happens to be a multiple of $(\log^{\fu})^6(\Li_2^{\fu}-\frac{1}{2}\log^{\fu}\Li_1^{\fu})$. The function $F^{\vert 2\vert}_{6,18}$ satisfying 
\[
\mathrm{Res}(\nu_4(X), \nu_6(X))=(\log^{\fu})^6\bigg(\Li_2^{\fu}-\frac{1}{2}\log^{\fu}\Li_1^{\fu}\bigg) F^{\vert 2\vert}_{6,18}
\]
is a non-trivial Chabauty--Kim function of depth $6$ and degree $18$ (\Cref{prop:F5}).
It is worth pointing out that the factor $\Li_2^{\fu}-\frac{1}{2}\log^{\fu}\Li_1^{\fu}$ is precisely the generator of $\mathcal{I}^{\PL}_{\{ \ell \},2}$ studied in \cite{DCW2}, where $\ell$ denotes any prime. We note however that it is not an element of $\mathcal{I}^{\PL}_{S,2}$ for $\vert S\vert =2$.

\subsection{Prior works}

The present work is the first to tackle explicit Chabauty--Kim in depth $> 4$ for $\vert S\vert\geq 2$. 
In this section, we give an overview of prior works on the subject.

The map $\loc$ was first made explicit in \cite{DCW1} in depth $2$. In \cite{BDCKW}, the case $S=\emptyset$ in depths $1$ and $2$ was studied. In \cite{DCW1}, the equality $X(\Z_S)=(X(\Z_p)^{\mathrm{PL}}_{S,d})^{S_3}$ was verified for $S=\{ 2\}$ and $d=2$ with $p=3,5,7$, but turned out to fail with $p=11$. In \cite{DCW2}, Dan-Cohen and Wewers studied the case $S=\{ 2\}$ in depth $4$ and verified the equality $X(\Z_S)=(X(\Z_p)^{\mathrm{PL}}_{S,d})^{S_3}$ with $p=3, 5, \ldots, 29$ using two explicit $p$-adic Chabauty--Kim functions in depths $2$ and $4$ respectively \cite[Thm.~1.17.1]{DCW2}. The motivic framework of the method was further developed in \cite{DC}, resulting in an algorithm for determining $X(\Z_S)$ whose halting is conditional on a list of conjectures. The polylogarithmic motivic Chabauty--Kim method for $X$ was fully developed in \cite{CDC1}, including the separation into the geometric and the arithmetic step. The geometric step was solved for $\vert S\vert=1$ in depth $4$, and the equality $X(\Z_S)=(X(\Z_p)^{\mathrm{PL}}_{S,d})^{S_3}$ was verified for $S=\{ 3 \}$ in depth $4$ with $p=5,7$. It is also in \cite{CDC1} that the necessity of symmetrising the polylogarithmic Chabauty--Kim locus via the $S_3$-action was discovered. In \cite{CDC2}, the authors describe an algorithm to determine the set $X(\Z_S)$, based on the method of \cite{CDC1}, whose halting depends on a list of conjectures. Part of the work of the present paper can be seen as making their algorithm for the geometric step effective, a step that was perceived in \cites{CDC1, CDC2} as simple but is far from it as soon as the size of $S$ or $d$ gets large.
In related work, motivic Chabauty--Kim was recently developed for the surface $M_{0,5}$ in \cite{DCJ} and was worked out explicitly for $S=\{ 2, 3\}$ in depth $4$. A certain weight advantage encountered for $M_{0,5}$ allowed Dan-Cohen and the first author to work in depth $4$, a depth that is too small to handle $X$ when $\vert S\vert =2$. Their calculations solve in particular the arithmetic step for $X$ in depth $4$ for $S=\{ 2, 3\}$. Recently, the case $S=\{ 2, 3\}$ for $X$ was tackled in \cite{martin} using a different method, namely the {\it refined Chabauty--Kim method} first developed in \cite{bettsdogra}. This refined version of the method allows L\"udtke to work in depth $4$, resulting in new evidence for the {\it refined Kim conjecture} in this setting. The previous work \cite{BB+} concerned the refined Chabauty--Kim method for $\vert S\vert=2$ in depth $2$. The recent work \cite{BKL24} proves Kim's conjecture for $S = \emptyset$ and the refined Kim conjecture for $S=\{2\}$, for all odd primes $p$. 

\subsection{Outline}

\Cref{s:mot} contains preliminaries on mixed Tate motives and motivic fundamental groups. \Cref{s:overview} gives an overview of the polylogarithmic motivic Chabauty--Kim method. \Cref{s:coord} gives a description of the evaluation map \eqref{intro:ev} in terms of non-canonical coordinates of the Selmer scheme. This marks the end of the background material. \Cref{s:algg} contains the upper bound algorithm and the proofs of \Cref{thm:intro}$(i)$, $(iii)$. Finally, in \Cref{chapter:resultants}, we describe the resultant method, which culminates in the proof of \Cref{thm:intro}. \Cref{s:multiple} contains a brief discussion of possible generalisations of the method to the case $\vert S\vert>2$.

\subsection*{Acknowledgements}

The authors thank David Corwin and Ishai Dan-Cohen for their support and encouragement throughout the project. They also thank Martin L\"udtke for numerous helpful conversations, and both Martin L\"udtke and Alexander Betts for their comments and corrections on a preliminary version of this paper, as well as the anonymous referees. DJ and SZ were supported by the Israel Science Foundation (grants No.\ 726/17 and No.\ 621/21 respectively, PI: Ishai Dan-Cohen) while at Ben-Gurion University of the Negev. DTBGL was supported by an Emily Erskine Endowment Fund Post-Doctoral Researcher Fellowship and the Israel Science Foundation (grant No.\ 2301/20, PI: Ari Shnidman) while at the Hebrew University of Jerusalem, and by an Edixhoven Post-Doctoral Fellowship while at Leiden University. FMS and AW were supported by the Israel Science Foundation (grant No.\ 1963/20, PI: Daniel Disegni) and by the US-Israel Binational Science Foundation (grant No.\ 2018250, PI: Daniel Disegni) while at Ben-Gurion University of the Negev.

%%%%%%%%%%%%%%%%%%%%%%%%%%%%%%%%%%%%%%%%%%%%%%%%%%%%%%%%%%%

\section{Mixed Tate motives and fundamental groups}\label{s:mot}

\subsection{Mixed Tate Motives}

Let $S$ be a finite set of primes, and let $\MT(\Q)$ denote the category of mixed Tate motives over $\Q$, as defined in \cite[Thm.~4.2]{levineTM} \footnote{Note that the hypothesis in \cite[Thm.~4.2]{levineTM} holds: the Beilinson--Soul\'e vanishing conjecture for number fields follows from work of Borel \cite{borel}.}: it is a rigid abelian tensor category whose objects are iterated extensions of the pure Tate motives $\Q(n)$, with $n\in \Z$.

\begin{definition}
    The category $\MT(\Z_S)$ is the full subcategory of $\MT(\Q)$ consisting of mixed Tate motives unramified outside $S$, in the sense of \cite[\S 1.7, Def.~1.4]{delignegoncharov}.
\end{definition}

Each object $Y$ of $\MT(\Z_S)$ has a functorial finite increasing filtration
\[
0\sub W_a(Y)\sub W_{a+1}(Y)\sub \cdots \sub W_{b-1}(Y) \sub W_b(Y)=Y,
\]
where, for each integer $n$, the quotient $\GR_n^W(Y) = W_{n}(Y)/W_{n-1}(Y)$ is isomorphic to a direct sum of a finite number of copies of $\Q(-n/2)$ (we adopt the convention that $\Q(-n/2)=0$ if $n$ is odd). The above filtration is called the \emph{weight filtration}.

The subcategory of semisimple objects of $\MT(\Z_S)$ consists of direct sums of the Tate objects, and is therefore isomorphic to the category $\GrVec_\Q$ of finite-dimensional graded $\Q$-vector spaces.
Moreover, by \cite[Cor. 4.3]{levineTM}, for all integers $p, q$, we have
\begin{equation}\label{eq:k-groups-extensions}
\mathrm{Ext}_{\MT(\Q)}^{p}(\Q(0), \Q(q))\simeq K_{2q-p}^{(q)}(\Q).
\end{equation}
The weight filtration induces a (canonical) graded piece functor
\begin{equation*}\label{funct:fibre}
\omega\colon \MT(\Z_S) \to \GrVec_\Q
\end{equation*}
given by
\[
Y \mapsto \bigoplus_{n\in \Z} (\GR^W_n(Y) \otimes \Q(n/2))=\bigoplus_{d\in \Z} \mathrm{Hom}_{\MT(\Z_S)}(\Q(-d), \mathrm{gr}^W_{2d}(Y))=:\bigoplus_{d\in \Z} \omega_d(Y).
\]
The degree $d$ of $\omega_d(Y)$ is referred to as the {\it half-weight}, but we will simply use the terminology ``degree''.
The functor $\omega$ is faithful and exact, and endows $(\MT(\Z_S), \omega)$ with the structure of a neutral $\Q$-Tannakian category (we refer to \cite{Rivano2} for details on the Tannakian formalism).

\begin{remark}\label{s:remarkgrading}
We have departed from the conventions of \cite{delignegoncharov} in order to maintain the conventions of \cite{CDC1}. More precisely, our $d\th$ graded piece $\omega_d(Y)$ of $\omega(Y)$ is Deligne--Goncharov's $(-d)\th$ graded piece. A consequence is that the pro-unipotent Lie algebras that we encounter later are \emph{negatively} graded.
\end{remark}

\subsection{The pro-unipotent mixed Tate fundamental group}\label{s:proT}

\begin{definition}
    Let 
    $G_S:=\Aut^\tensor(\omega)$
be the Tannakian fundamental group of $(\MT(\Z_S), \omega)$. It is a (pro-)affine algebraic group over $\Q$ and $\omega$ induces an equivalence of categories $\MT(\Z_S) \simeq \mathrm{Rep}_{\Q}(G_S)$.
\end{definition}

The action of $G_S$ on 
\[
\omega(\Q(-1))=\omega_1(\Q(-1))=\Hom_{\MT(\Z_S)}(\Q(-1),\Q(-1))=\Q,
\]
gives a non-trivial map from $G_S$ to $\GL_{1,\Q}$, and thus a surjection $G_S\twoheadrightarrow\G_m$.

\begin{definition}\label{def:pro-unip}
    The \emph{pro-unipotent fundamental group} $U_S$ of $\MT(\Z_S)$ is the kernel of $G_S\twoheadrightarrow\G_m$. 
\end{definition}

By definition, there is a short exact sequence
\[1\to U_S\to G_S\to\G_m\to 1\]
of affine algebraic groups over $\Q$. This short exact sequence splits: 
letting $\lambda\in \G_m$ act on $\omega_d(X)$ by multiplication by $\lambda^d$ gives a $\G_m$-action (i.e., a grading) on $\omega(X)$, and hence a splitting $\G_m\to G_S$ (known as the Levy splitting of $G_S$). 
Thus, there is an isomorphism
\[G_S\simeq U_S\rtimes\G_m.\]
The action of $U_S$ respects the weight filtration on $\MT(\Z_S)$ and is trivial on weight graded pieces. It follows that $U_S$ is the maximal pro-unipotent subgroup of $G_S$ \cite[\S 2.1]{delignegoncharov}.

We may describe $U_S$ in terms of the extension groups $\Ext^1(\Q(0), \Q(n))$ in $\MT(\Z_S)$ (the $\Ext^2$ are trivial in $\MT(\Z_S)$ \cite[Prop.~1.9]{delignegoncharov}), as we will now explain. \footnote{To avoid any ambiguity when taking duals of graded vector spaces, we adopt the conventions outlined in \cite[\S 2.1.1]{CDC1}.} By \eqref{eq:k-groups-extensions}, we have
\[
\Ext^1(\Q(0), \Q(n))=K_{2n-1}^{(n)}(\Z_S)=
\begin{cases}
K_{2n-1}(\Z_S)\otimes \Q & n\geq 0 \\
0 & n<0.
\end{cases}
\]
When $n = 1$, we have $K_1(\Z_S)\otimes \Q=\Z_S^\times \otimes \Q$, which has dimension $\vert S\vert$. When $n\ge 2$, we have $K_{2n-1}(\Z_S)\otimes \Q=K_{2n-1}(\Q)\otimes \Q$, which has dimension $0$ for $n$ even and dimension $1$ for $n$ odd \cite{borel}. 

By the formalism of pro-unipotent groups, it follows that $\mathrm{Lie}(U_S)$ is a free, negatively graded, pro-nilpotent Lie algebra, which is finite-dimensional in each degree (the grading is negative because of \Cref{s:remarkgrading}). Moreover, the functor $\omega$ induces an equivalence between $\MT(\Z_S)$ and the category of graded finite-dimensional $\Q$-vector spaces with an action of $\Lie(U_S)$ compatible with the gradings \cite[Prop.~2.2]{delignegoncharov}. By \cite[A.15]{delignegoncharov}, a system of homogeneous free generators of $\Lie(U_S)$ is obtained by (non-canonically) lifting a basis of each 
\[
\Lie(U_S)^{\ab}_{-n}=\Ext^1(\Q(0), \Q(n))^\vee.
\]
Hence, the pro-unipotent group $U_S$ is non-canonically isomorphic to the free pro-unipotent group (in the sense of \cite[\S 2.1.3]{CDC1}) on the graded vector space $\bigoplus_{n\in \Z} \Ext^1(\Q(0), \Q(n))^\vee$ with $\Ext^1(\Q(0), \Q(n))^\vee$ placed in degree $-n$.

\begin{definition}\label{def:free-generators-n}
    Fix once and for all a system of free generators of $\Lie(U_S)$ given by $\{\tau_\ell \}_{\ell\in S}$  in degree (i.e., half-weight) $-1$, and $\sigma_{2n+1}$ in degree $-2n-1$, for each $n\ge 1$. We denote by $\Sigma$ this ordered set of free generators, and by $\Sigma_{-d}$ the subset of generators of degree $-d$. \footnote{These generators may be chosen arbitrarily, but for the arithmetic step one should choose them as in \cite[Prop.~4.1]{CDC1}.} 
\end{definition}

As a Hopf algebra, $\oh(U_S)$ is dual to the completed universal enveloping algebra $\UU\Lie(U_S)$. Hence, by the formalism of \cite[\S2.1.3]{CDC1}, it follows that $\oh(U_S)$ is non-canonically isomorphic to the free shuffle algebra on the positively graded vector space
$\bigoplus_{n=1}^\infty \Ext^1(\Q(0), \Q(n))$ with $\Ext^1(\Q(0), \Q(n))$ placed in degree $n$. 

Explicitly, for each word $w\in \UU\Lie(U_S)$ in the fixed generators $\{ \tau_\ell \}_{\ell\in S}$ and $\{ \sigma_{2n+1} \}_{n\geq 1}$ of $\Lie(U_S)$, we denote the dual basis element by $f_w\in \oh(U_S)$. The coproduct is the deconcatenation coproduct given by 
\begin{equation}\label{eq:coprod}
\Delta f_w:= \sum_{w_1w_2=w} f_{w_1}\otimes f_{w_2}.
\end{equation}
For words $w_1, w_2\in \UU\Lie(U_S)$, the commutative product is the shuffle product given by 
\[
f_{w_1} f_{w_2}:= \sum_{\sigma \in \Sha(\ell(w_1), \ell(w_2))} f_{\sigma(w_1w_2)},
\]
where $\Sha(\ell(w_1), \ell(w_2))\subseteq S_{\ell(w_1)+\ell(w_2)}$ is the subgroup of shuffle permutations of type $(\ell(w_1), \ell(w_2))$. 

The elements $\{ f_w \}_w$ constitute a basis of $\oh(U_S)$, which we call the \emph{abstract shuffle basis}.

\begin{definition}\label{def:az-le-d}
    We write $\oh(U_S)_{\leq d}$ for the subalgebra of $\oh(U_S)$ generated by words of degree at most $d$. More precisely, $\oh(U_S)_{\leq d}=\oh((U_S)_{\geq -d})$, where $(U_S)_{\geq -d}$ is the pro-unipotent group associated to the quotient Lie algebra $\Lie(U_S)/\Lie(U_S)_{< -d}$ generated by the set $\Sigma_{\geq -d}$ of generators in $\Sigma$ of degree $\geq -d$.
\end{definition}

\subsection{The unipotent de Rham fundamental group}\label{s:dRfund}

Let 
$X := \P^1 \setminus \{0,1,\infty\}$ denote the thrice-punctured projective line,
viewed as a scheme over $\Z_S$. 

\begin{definition}
    Fix a point $x\iQ\setminus \{ 0,1 \}$ or a tangential vector at $0, 1$ or $\infty$. The \emph{unipotent de Rham fundamental group} of $X$ at $x$, $\pi^{\mathrm{un}, \dR}_1(X; x)$, is the Tannakian fundamental group of the neutral $\Q$-Tannakian category of algebraic vector bundles with nilpotent connection on $X_\Q$, with fibre functor given by taking the fibre at $x$.
    In the case that $x=\vec{1}_0$ is the tangential vector $1$ at the point $0$, we set 
\[
\Pi:=\pi^{\mathrm{un}, \dR}_1(X;\vec{1}_0). \footnote{Here we do not follow the notation in Deligne--Goncharov \cite[\S 5.7]{delignegoncharov}. Namely, our $\Pi$ is their $\Pi_{0,0}$, but it is isomorphic to what Deligne and Goncharov denote by $\Pi$.}
\]
\end{definition}

The group $\Pi$ is an affine, pro-unipotent group over $\Q$. Its coordinate ring $\oh(\Pi)$ is the de Rham realisation of a commutative Hopf algebra in the category of $\mathrm{ind}$-objects of $\MT(\Z)$, the category of unramified mixed Tate motives over $\Q$ \cite[\S4, \S5]{delignegoncharov}. The de Rham realisation functor of $\MT(\Z_S)$ is canonically isomorphic to $\omega$ \footnote{Because the Hodge filtration splits the weight filtration.} \cite[\S2.9]{delignegoncharov}. In particular, the group $\Pi$ carries an action of $G_S$.

Let $V$ denote the graded vector space consisting of $H_1^{\dR}(X_\Q)$ in degree $-1$ and $0$ in all other degrees. Then we can view $\Pi$ as the free pro-unipotent group on $V$, in the following sense. Take the completed graded tensor algebra $TV$ on $V$ and put the unique coproduct on it such that all elements of $V$ are primitive. Denote by $\mathfrak{n}$ the subspace of $TV$ consisting of primitive elements. Then $\mathfrak{n}$ has the structure of a strictly negatively graded pro-nilpotent Lie algebra, so it is the Lie algebra of a graded pro-unipotent group $U(V)$. The statement that $\Pi$ is the free pro-unipotent group on $V$ means that $\Pi=U(V)$. In particular, the coordinate ring $\oh(\Pi)$ is the Hopf dual algebra of the universal enveloping algebra $\mathcal{U}\mathfrak{n}$. This enveloping algebra is $TV$ with the unique coproduct for which all elements of $H_{1}^{\dR}(X_\Q)$ are primitive. 

This description allows us to apply the machinery of \cite[\S 2.1.3]{CDC1}. 
Take $\{ e_0, e_1 \}\sub TV$ to be the basis dual to the standard basis $\{ dz/z, dz/(1-z) \}$ of $H^{1}_{\dR}(X_\Q)$, assign to these basis elements the degree $-1$ and declare them to be primitive. Then $TV$ is generated by words $w$ in the basis elements $e_0$ and $e_1$. Let $\Li^{\mathfrak{u}}_w$ denote the element of $\oh(\Pi)$ that is dual to the word $w$. Then $\oh(\Pi)$ is generated by the elements $\{ \Li^{\mathfrak{u}}_w \}_w$.
The deconcatenation coproduct is given by  
\begin{equation}\label{concat}
\Delta \Li^{\mathfrak{u}}_w := \sum_{w_1w_2=w} \Li^{\mathfrak{u}}_{w_1}\otimes \Li^{\mathfrak{u}}_{w_2},   
\end{equation}
and the shuffle product by
\begin{equation}\label{shuffle}
\Li^{\mathfrak{u}}_{w_1} \Li^{\mathfrak{u}}_{w_2} := \sum_{\sigma \in \Sha(\ell(w_1), \ell(w_2))} \Li^{\mathfrak{u}}_{\sigma(w_1w_2)}.
\end{equation}

%%%%%%%%%%%%%%%%%%%%%%%%%%%%%%%%%%%%%%%%%%%%%%%%%%

\section{The polylogarithmic motivic Chabauty--Kim method}\label{s:overview}

We restrict the discussion exclusively to the polylogarithmic quotient of $\Pi$. For the more general situation, as well as details, see \cite[\S 2]{CDC1}.

\subsection{The polylogarithmic quotient}\label{sec:polylog quotient}

Recall that 
$X = \P^1 \setminus \{0,1,\infty\}$ is viewed as a scheme over $\Z_S$.
The inclusion of schemes $X\hookrightarrow \G_m$ gives rise to a $G_S$-equivariant morphism $\Pi\to \pi^{\mathrm{un}, \dR}_1(\G_m;\vec{1}_0)=\Q(1)$, the target group being the pro-unipotent de Rham fundamental group of $\mathbb{G}_m$ at the tangential base-point $\vec{1}_{0}$. Let $N$ denote the kernel of this homomorphism. 

\begin{definition}
    The \emph{polylogarithmic unipotent de Rham fundamental group} is defined as 
\[
\Pi^{\PL}:=\Pi/[N, N].
\]
\end{definition}
The group $N$ carries an action of $G_S$, hence $[N,N]$ is $G_S$-stable and, as a consequence, $\Pi^{\PL}$ carries an action of $G_S$.
Deligne \cite[Prop.~16.13]{del} proved that
\[
\Pi^{\PL}=\Q(1)\ltimes \prod_{i=1}^\infty \Q(i),
\]
the action of $\Q(1)$ on the infinite product being the one defined by \cite[\S 16.12]{del}. In particular, the action of $G_S$ factors through $\G_m$.
Given $d\geq 1$, the depth $d$ polylogarithmic quotient is $\Pi^{\PL}_{\geq -d}$ (in the notation of \cite[p. 1869]{CDC1}). 

Recall from \Cref{s:dRfund} that $\Pi$ is the free pro-unipotent group on the graded vector space $V$, which is $H_1^{\dR}(X_\Q)=\mathrm{Span}_{\Q}\{ e_0, e_1 \}$ in degree $-1$ and zero elsewhere. The specific words of the form
\[
e_0, \quad e_1, \quad e_1e_0, \quad e_1e_0e_0, \quad e_1e_0e_0e_0, \quad\ldots
\]
are called polylogarithmic words, and we denote their corresponding dual basis elements in $\oh(\Pi)$  by 
\[
\log^{\mathfrak{u}}, \quad \Li_1^{\mathfrak{u}}, \quad \Li_2^{\mathfrak{u}}, \quad \Li_3^{\mathfrak{u}}, \quad\Li_4^{\mathfrak{u}}, \quad \ldots.
\]
We then have  
\[\O(\Pi^{\PL}) = \langle \log^\uu, \Li_1^\uu, \Li_2^\uu,\ldots\rangle\sub \O(\Pi).\] 
For each $d\ge 1$, let $\O(\Pi^{\PL})_{\leq d}\sub\O(\Pi^{\PL})$ be the coordinate ring of $\Pi^{\PL}_{\geq -d}$. Then 
\[\oh(\Pi^{\PL})_{\leq d}=\Q[\log^{\fu}, \Li_1^{\fu}, \Li_2^{\fu}, \ldots, \Li_d^{\fu}]\]
is a Hopf subalgebra of $\O(\Pi^{\PL})$.

\subsection{The Selmer scheme}

\begin{definition}\label{def:ev}
    The polylogarithmic motivic Selmer scheme is defined as
    \[
\Sel_S^{\PL}:=\Hom_{\G_m}(U_S, \Pi^{\PL}).
\] 
It is a $\Q$-scheme that comes equipped (by representability) with a universal evaluation map (a morphism of $U_S$-schemes)
\[
\ev \: \Sel_S^{\PL} \times_\Q U_S \to \Pi^{\PL} \times_\Q U_S.
\] 
\end{definition}
By \cite[Prop.~5.2.1]{DCW2}, the set of $\Q$-points of the Selmer scheme can be identified with 
\[
\Sel_S^{\PL}(\Q)=H^1(G_S, \Pi^{\PL}):=\{ \Pi^{\PL}\text{-torsors over } \Q \text{ with a compatible action of } G_S\}.
\]

Let $K := \operatorname{Frac}(\oh(U_S))$ denote the function field of $U_S$. 
Pull back the morphism of $U_S$-schemes $\ev$ to a map of $K$-schemes 
\[
\ev_K : (\Sel_S^{\PL})_{K} \to (\Pi^{\PL})_{K}. 
\]

\begin{definition}\label{def:I}
    Define
    \[
\cI^{\PL}_S:=\ker(\ev_K^{\#})=\{ f\in \cO((\Pi^{\PL})_{K}) \mid f(\im(\ev_K))=0 \},
\]
where $\ev_K^{\#}$ is the map on coordinate rings induced by $\ev_K$.
\end{definition}

\subsection{The Chabauty--Kim diagram}\label{s:diag}

Let $p \not\in S$ be a prime.
The (expanded) Chabauty--Kim diagram is a commutative diagram of sets
\begin{equation}\label{diag1}
\begin{tikzcd}[bend angle = 80]
  X(\Z_S) \arrow[hook]{r} \arrow{d}{\kappa} & X(\Z_p) \arrow{d}{\kappa_p} \\
  H^1(G_S, \Pi^{\PL}) \arrow[equal]{d} \arrow{r}{\loc} & \Pi^{\PL}(\Q_p) \arrow[equal]{dd} \\
  \Sel_S^{\PL}(\Q) \arrow[hook]{d} & \\
  \Sel_S^{\PL}(\Q_p) \arrow{r}{\loc} & \Pi^{\PL}(\Q_p),
\end{tikzcd}
\end{equation}
with maps given as follows:
\begin{itemize}[leftmargin=*]
    \item {\it The motivic Kummer map $\kappa$}: Given a point $b\in X(\Z_S)$, assign to it a $G_S$-equivariant $\Pi$-torsor over $\Z_S$, namely the torsor of paths from $b$ to $\vec{1}_0$ denoted $_bP_{\vec{1}_0}$. This torsor can naturally be viewed as an element of $H^1(G_S, \Pi)$. Now, $\kappa(b)$ is the image of $_bP_{\vec{1}_0}$ in $H^1(G_S, \Pi^{\PL})$. 
    
    \item {\it The local Kummer map $\kappa_p$}: Given a point $b\in X(\Z_p)$, assign to it the point $\varphi_b \in   \Pi(\Q_p)$ determined by
    \[
    \varphi^{\#}_b \: \cO(\Pi)\to \Q_p, \qquad \Li^{\fu}_w \mapsto \int_b^{\vec{1}_0} \omega_w,
    \]
    where $\omega_w$ is the sequence of differential forms associated with the word $w$.  
    Here the integral is taken in the sense of Coleman's theory of $p$-adic iterated integrals. Its image in the quotient $\Pi^{\PL}$ is then $\kappa_p(b)$. 
    The map $\kappa_p$ thus defined is $p$-adic locally analytic. 
    
    \item {\it The map $\loc$:} There is a canonical $\Q_p$-point $\per_p$ of $U_S$, or equivalently, a map 
    $
    \per_p^{\#} \: \oh(U_S) \to \Q_p,
    $
    known as the $p$-adic period map (the point $(\eta_p^{\mathrm{ur}})^{-1}$ in the notation of \cite{perp}, see \cite[\S 3.3]{BKL24}). It maps polylogarithmic values to the corresponding $p$-adic polylogarithmic values. 
    Pulling back the evaluation map $\ev$ along $\per_p \: \spec(\Q_p)\to U_S$ gives a morphism of $\Q_p$-schemes
    \[
    \ev_{\Q_p} \: \Sel^{\PL}_S \times_{\per_p} \spec(\Q_p) \to \Pi^{\PL} \times_{\per_p} \spec(\Q_p),
    \]
    which in turn induces a map on $\Q_p$-points
    \begin{align*}
    \loc_{\Q_p} \: \Sel^{\PL}_S(\Q_p) &\to \Pi^{\PL}(\Q_p)\\
    c&\mapsto c(\per_p).
    \end{align*}
    The map $\loc$ in \eqref{diag1} is the restriction of the map $\loc_{\Q_p}$ to the $\Q$-points $\Sel_S^{\PL}(\Q)$. In particular, $\loc$ is algebraic, in the sense that it arises as the map on $\Q_p$-points of the morphism of schemes $\ev_{\Q_p}$.
\end{itemize}

\subsection{The Chabauty--Kim locus}

The goal is to determine the finite set $X(\Z_S)$. Denote the pull-back of $\Pi^{\PL}\times_{\Q} U_S \to U_S$ via $\per_p \colon \spec(\Q_p)\to U_S$ by $\Pi^{\PL}_{\Q_p}\to \spec(\Q_p)$. It comes equipped with a morphism $\varphi \colon \Pi^{\PL}_{\Q_p}\to \Pi^{\PL}\times_{\Q} U_S$.
Given $f\in \O(\Pi^{\PL}\times U_S)$, we let $f_{\Q_p} = \varphi^{\#}(f)\in \oh(\Pi^{\PL}_{\Q_p})$. The latter induces a function $f_{\Q_p} \colon \Pi^{\PL}(\Q_p)\to \Q_p$ and we obtain, following \cite[Def.~2.26]{CDC1}, a $p$-adic Chabauty--Kim function by taking the composition
\begin{equation}\label{pCK}
f|_{X(\Z_p)} \: X(\Z_p)\xrightarrow{\kappa_p}\Pi^{\PL}(\Q_p)\xrightarrow{f_{\Q_p}} \Q_p.
\end{equation}

Recall from \Cref{def:I} that $\cI^{\PL}_S=\{ f\in \cO((\Pi^{\PL})_K) \mid f(\im(\ev_K))=0 \}$. Given an integer $d\geq 1$, we define 
\begin{equation}
    \mathcal{I}^{\PL}_{S,d}:=\O(\Pi^{\PL}_{\geq -d}\times (U_S)_{\geq -d})\cap \cI^{\PL}_S.
\end{equation}

\begin{definition}
    Given an integer $d\geq 1$, we define the \emph{polylogarithmic motivic Chabauty--Kim locus} in depth $d$ to be
    \[X(\Z_p)^{\PL}_{S, d} := \set{z\in X(\Z_p) : f|
_{X(\Z_p)}(z) = 0,\ \ \forall f\in \mathcal{I}^{\PL}_{S,d}}. \footnote{Note that our definition of the ideal $\mathcal{I}^{\PL}_{S,d}$ differs slightly from \cite[Def. 2.22]{CDC1} since they allow $K$-coefficients. However, our definition of $X(\Z_p)^{\PL}_{S, d}$ remains the same as \cite[Def. 2.27]{CDC1}. The advantage of allowing $K$-coefficients is described in \cite[Rem. 2.21]{CDC1} and we use $K$-coefficients in the elimination theory performed in Section \ref{chapter:resultants}, although we eventually clear denominators to produce elements in $\mathcal{I}^{\PL}_{S,d}$.} 
\] 
\end{definition}
From the commutativity of \eqref{diag1}, the set $X(\Z_S)$ is contained in $X(\Z_p)^{\PL}_{S, d}$ for all $d$. 
We thus obtain a descending sequence of sets
\[
X(\Z_p)\supseteq X(\Z_p)^{\mathrm{PL}}_{S,1} \supseteq X(\Z_p)^{\mathrm{PL}}_{S,2} \supseteq \cdots \supseteq X(\Z_p)^{\mathrm{PL}}_{S,d} \supseteq \cdots \supseteq X(\Z_S).
\] 
Moreover, the set $X(\Z_p)^{\mathrm{PL}}_{S,d}$ is finite for sufficiently large $d$ (this follows for instance from \Cref{prop:dpl-dphi} below).
It is conjectured \cite[Conj. 2.32]{CDC1} that 
\begin{equation}\label{conj:cdc}
X(\Z_S)=(X(\Z_p)^{\mathrm{PL}}_{S,d})^{S_3}, \qquad \text{ for } d\gg 0.  
\end{equation}
There are inclusions
\[
X(\Z_S)\subseteq X(\Z_p)^{\rm Kim}_{S,d}\subseteq (X(\Z_p)^{\mathrm{PL}}_{S,d})^{S_3},
\]
where $X(\Z_p)^{\rm Kim}_{S,d}$ denotes Kim's $p$-adic Chabauty locus in depth $d$ (see \cite[Rem. 2.28]{CDC1}).
From these inclusions, it is clear that conjecture \eqref{conj:cdc} implies Kim's conjecture \cite{kim}. One insight of \cite{CDC1} is that it may suffice to use single polylogarithms to deal with the $S$-unit equation, as opposed to requiring multiple polylogarithms as in Kim's original approach.
Dan-Cohen and Corwin verified the conjectural equality \eqref{conj:cdc}  for $S=\{3 \}$, $p=5,7$, and $d=4$ \cite[Thm. 5.5]{CDC1}.

The main advantage of the motivic approach of \cites{CDC1, CDC2} is the division of the computation of $X(\Z_p)^{\mathrm{PL}}_{S,d}$ into two steps: 
\begin{itemize}[leftmargin=*]
\item {\it Geometric step in depth $d$}: the task of finding non-trivial polylogarithmic motivic Chabauty--Kim functions in depth $\leq d$, i.e., non-zero elements $f\in \mathcal{I}^{\PL}_{S,d}$. These are polynomials in $\oh(U_S)_{\leq d}[\log^{\fu}, \Li_1^{\fu}, \ldots, \Li_d^{\fu}]$. This step depends on the cardinality of $S$, but not on the specific primes contained in $S$, nor on the prime $p$;
\item {\it Arithmetic step in depth $d$}: the task of determining the $p$-adic Chabauty--Kim functions $f|_{X(\Z_p)}$ defined by \eqref{pCK} for the functions $f$ found in the geometric step. This requires the calculation of the image of $f$ under the $p$-adic period map $\per_p$. The resulting functions $f|_{X(\Z_p)}$ can be checked to be non-trivial in practice (and are guaranteed to be non-trivial if one assumes the $p$-adic Period Conjecture \cite[Conj. 2.25]{CDC1}). This step depends on the prime $p$ and the specific primes contained in $S$. 
\end{itemize}

\section{The geometric step in coordinates}\label{s:coord}

The goal of the geometric step in depth $d$ is to find non-zero elements 
\[
f\in \mathcal{I}^{\PL}_{S,d}=\O(\Pi^{\PL}_{\geq -d} \times (U_S)_{\geq -d})\cap \mathcal{I}^{\PL}_S,
\] 
where we recall that the ideal $\mathcal{I}^{\PL}_S=\ker(\ev_K^{\#})$ is defined in \Cref{def:I}. 
In this section we describe, following \cite{CDC1}, a (non-canonical) affine model for the Selmer scheme $\Sel_S^{\PL}$ (\Cref{def:ev}) and give an explicit formula for the map $\ev$ in the chosen coordinates.  

\subsection{Coordinates on Selmer schemes}

Recall from \Cref{def:free-generators-n} the fixed system of free generators $\Sigma = \{\tau_\l, \sigma_{2n+1}\}$ of $\Lie(U_S)$, with their specified degrees. Each word $w$ in these generators corresponds to a dual element in $\oh(U_S)$, denoted $f_w$ (the abstract shuffle basis). 

Let $R$ be a $\Q$-algebra, $c\in \Sel^{\PL}_S(R)$, and $\lambda$ be a polylogarithmic word of degree $d$. 
We may write $\Li_\lambda^{\fu}(c):=c^{\#}(\Li_\lambda^{\fu})\in \oh(U_S) \otimes R$ in terms of the abstract shuffle basis:
\[
\Li_\lambda^{\fu}(c)=\sum_{\deg(w)=-d} \phi_{\lambda}^{w}(c) f_w \in \oh(U_S)\otimes R,\]
with coefficients $\phi_{\lambda}^{w}(c)\in R$.
The sum is taken over words of degree $d$ since $c^{\#}$ preserves the gradings by $\mathbb{G}_m$-equivariance.
For each polylogarithmic word $\lambda$ and for each $\rho\in \Sigma$ of the same degree, introduce a free variable $\Phi_\lambda^\rho$. 
Define the polynomial algebra
\[\Q[\Phi]:=\Q[\{ \Phi_{\lambda}^{\rho} \}].\]
Then the assignment
\[c \mapsto (\phi_\lambda^\rho(c))_{\lambda, \rho}\]
defines a map
\[\Psi\: \Sel^{\PL}_S\to\Spec(\Q[\Phi]),\]
which, by \cite[Cor.~3.11]{CDC1}, is an isomorphism of $\Q$-schemes. This isomorphism provides a (non-canonical) description of $\Sel^{\PL}_S$ as an affine scheme. 
Moreover, by \cite[\S 3.3.1]{CDC1}, there are compatible isomorphisms for each integer $d$
\begin{equation}\label{isom:psin}
    \Psi_d \: \Sel^{\PL}_{S,d} \xrightarrow{\sim}\Spec(\Q[\Phi,d])
\end{equation}
where $\Sel^{\PL}_{S,d}:=\Hom_{\G_m}((U_S)_{\geq -d}, \Pi^{\PL}_{\geq -d})$ and $\Q[\Phi,d]=\Q[\{ \Phi_{\lambda}^{\rho} \}_{\deg(\rho)=\deg(\lambda)\geq -d}]$.

\subsection{The evaluation map in coordinates}
Observe that 
\[
\Sel^{\PL}_S \times_{\Q} U_S \overset{\Psi\times \id}{\simeq} \Spec(\Q[\Phi])\times_\Q U_S=\Spec(\oh(U_S)[\Phi]).
\]
\begin{definition}\label{theta}
    Define a map of $U_S$-schemes
\[
\theta \: \Spec(\oh(U_S)[\Phi]) \to \Pi \times_\Q U_S,
\]
or equivalently, a map of $\oh(U_S)$-algebras
\[
\theta^{\#} \: \oh(U_S)[\log^{\fu}, \Li_1^{\fu}, \Li_2^{\fu}, \ldots] \to \oh(U_S)[\Phi]
\]
by 
\begin{equation*}
\log^{\fu} \mapsto \sum_{\tau\in \Sigma_{-1}} f_{\tau} \Phi_{e_0}^\tau
\qquad \text{ and } \qquad
\Li_k^{\fu} \mapsto \sum f_{\rho\tau_{(1)}\ldots \tau_{(r)}}\Phi^\rho_{e_1 e_0^{s-1}}\Phi_{e_0}^{\tau_{(1)}}\ldots \Phi_{e_0}^{\tau_{(r)}},
\end{equation*}
where the latter sum is taken over $\tau_{(1)}, \ldots, \tau_{(r)}\in \Sigma_{-1}$ and $\rho\in \Sigma_{-s}$ such that $r+s=k$ and $1\leq s\leq k$.
\end{definition}
 
By \cite[Cor. 3.11]{CDC1}, we have 
\[
\theta \circ (\Psi\times \id)=\ev.
\]
Moreover, everything remains valid in bounded depth. Hence, if $f\in \oh(U_S)_{\leq d}[\log^{\fu}, \Li_1^{\fu}, \Li_2^{\fu}, \ldots, \Li_d^{\fu}]$, then $f\in \cI_{S,d}^{\PL}$ if and only if $\theta^\#(f)=0$.

\subsection{The geometric step in coordinates}

In conclusion, the task of solving the geometric step in depth $d$ has been reduced to finding the kernel of the map of finite-dimensional $\oh(U_S)_{\leq d}$-algebras
\[
\theta^{\#}_d \: \oh(U_S)_{\leq d}[\log^{\fu}, \Li_1^{\fu}, \Li_2^{\fu}, \ldots, \Li_d^{\fu}] \to \oh(U_S)_{\leq d}[\Phi, d]
\]
defined in \Cref{theta}, a problem in linear algebra.

In \cite{CDC1}, the geometric step was solved in the case $\vert S\vert =1$ in depth $n=4$. They found \cite[Prop.~4.5]{CDC1} that the functions 
\begin{equation}\label{F122}
F^{\vert 1\vert}_{2,2}:=\Li_2^{\fu}-\frac{1}{2}\log^{\fu}\Li_1^{\fu},
\end{equation}
\[
F^{\vert 1\vert}_{4,4}:=f_{\sigma_3}f_{\tau_\ell}\Li_4^{\fu}-f_{\sigma_3\tau_\ell}\log^{\fu}\Li_3^{\fu}-\frac{(\log^{\fu})^3\Li_1^{\fu}}{24}(f_{\sigma_3}f_{\tau_\ell}-4f_{\sigma_3\tau_\ell})
\]
lie in $\mathcal{I}^{\PL}_{S, 4}$ for $S=\{ \ell \}$.

%%%%%%%%%%%%%%%%%%%%%%%%%%%%%%%%%%%%%%%%%%%%%%%%%%%%%%%%%%%%

\section{An upper bound algorithm}\label{s:algg}

Fix an integer $d \ge 1$. In the previous section, we reduced the geometric step to the computation of the kernel of the map of $\oh(U_S)_{\leq d}$-algebras
\[
\theta^{\#}_d \: \oh(U_S)_{\leq d}[\log^{\fu}, \Li_1^{\fu}, \Li_2^{\fu}, \ldots, \Li_d^{\fu}] \to R_{\Phi,d}\sub  \oh(U_S)_{\leq d}[\Phi, d]
\]
defined in \Cref{theta}, where
\[R_{\Phi,d}:= \oh(U_S)_{\leq d}\left[\{\Phi_{e_0}^{\tau_\l}, \Phi_{e_1}^{\tau_\l} : \l\in S\}\cup \{\Phi^{\sigma_{2n+1}}_{e_1e_0^{2n}} : n\ge 1\}\right].\]
Moreover, $\theta_d^\#$ is a homomorphism of graded rings, where: 
\begin{itemize}
\item each $\Li_i^\uu$ is assigned degree $i$, and $\log^\uu$ is assigned degree $1$;
\item each $\Phi^{\tau_\l}_{e_0}$ and $\Phi^{\tau_\l}_{e_1}$ is assigned degree $1$;
\item each $\Phi^{\sigma_{2n+1}}_{e_1e_0^{2n}}$ is assigned degree $2n+1$.
\end{itemize}
Hence, for each integer $v$, we may focus on the $v\th$ graded piece to obtain a map
\begin{equation}\label{theta-d-v}
    \theta_{d, v}^\#\:(\oh(U_S)_{\leq d}[\log^\uu, \Li_1^\uu, \ldots, \Li_d^\uu])_v \to (R_{\Phi, d})_v
\end{equation}
of free finite-rank $\oh(U_S)_{\leq d}$-modules. 
\begin{definition}\label{def:deg} 
    We call an element of $\ker(\theta^{\#}_{d, v})$ a {\it polylogarithmic motivic Chabauty--Kim function} of depth $d$ and degree $v$.
\end{definition}

\subsection{Existence of Chabauty--Kim functions in high depth}

Fix a depth $d$ and a degree $v$ in the sense of \Cref{def:deg}. Both $(\oh(U_S)_{\leq d}[\log^\uu, \Li_1^\uu, \ldots, \Li_d^\uu])_v$ and $(R_{\Phi, d})_v$ have natural monomial bases, so we can view $\theta_{d, v}^\#$ as a matrix $M(\theta_{d, v}^\#)$ with coefficients in $\oh(U_S)_{\leq d}$ with respect to these bases. Define
\[\dim_{\PL}(d, v):= \mathrm{rank}_{\oh(U_S)_{\leq d}}(\oh(U_S)_{\leq d}[\log^\uu, \Li_1^\uu, \ldots, \Li_d^\uu])_v\]
and 
\[\dim_\Phi(d, v):=\mathrm{rank}_{\oh(U_S)_{\leq d}} (R_{\Phi, d})_v.\]
Given a tuple of positive integers $\alpha$ of length $\vert\alpha\vert$ and a positive integer $n$, denote the number of partitions of $n$ with respect to $\alpha$ by $p(n; \alpha):=\# \{ \beta \in (\Z_{\geq 0})^{\vert\alpha\vert} \mid \langle \alpha, \beta\rangle=n \}$. We then have
\[\dim_{\PL}(d, v) = p(v; (1, 1, 2, \ldots, d))\]
and
\[\dim_\Phi(d, v) = p(v; (\underbrace{1, \ldots, 1}_{2|S| \text{ times}}, 3, 5, \ldots, \max(k\le d : k\text{ odd})).\]
The numerology stems from the fact that $\Phi_{e_0}^{\tau_\ell}$ and $\Phi_{e_1}^{\tau_\ell}$ all have degree $1$ for $\ell\in S$, while $\Phi^{\sigma_{2k+1}}_{e_1e_0^{2k}}$ has degree $2k+1$.

\begin{remark}\label{clever}
    The matrix $M(\theta_{d, v}^\#)$ is of size $\dim_\Phi(d, v) \times \dim_{\PL}(d, v)$. In particular, if $\dim_{\PL}(d, v) > \dim_\Phi(d, v)$, then its kernel is necessarily non-trivial, in which case there exists a polylogarithmic motivic Chabauty--Kim function in depth $d$ and degree $v$.
\end{remark}

\begin{proposition}\label{prop:dpl-dphi}
    For all $|S|$, there exist integers $d$ and $v$ such that $\dim_{\PL}(d, v) > \dim_\Phi(d, v)$.
\end{proposition}

\begin{proof}
Let us adopt the convention that $\dim_{\PL}(d,0) = \dim_{\Phi}(d,0) = 1$. Fix $|S|$ and $d$, and let $v$ be arbitrarily large. We have:
\[
\dim_{\PL}(d, v) = \sum_{r=0}^{[v/d]}\dim_{\PL}(d-1, v-rd) \gg v\dim_{\PL}(d-1, v/2),
\]
where we implicitly used the monotonicity in $v$ of $\dim_{\PL}(d, v)$. We may continue iteratively to obtain the bound:
\[
\dim_{\PL}(d, v) \gg v^{d-1}\dim_{\PL}(1, v/2^{d-1}) \gg v^d,
\]
where the last inequality is due to the fact that $\dim_{\PL}(1, n) = n + 1$.

On the other hand, if $d\ge 1$ is even, then $\dim_{\Phi}(d, v) = \dim_{\Phi}(d-1, v)$. Otherwise:
\[
\dim_{\Phi}(d, v) = \sum_{r=0}^{[v/d]}\dim_{\Phi}(d-2, v-rd) \ll v\dim_{\Phi}(d-2, v).
\]
Continuing iteratively, we reach the bound:
\[
\dim_{\Phi}(d, v) \ll v^{(d-1)/2}\dim_{\Phi}(1, v) \ll v^{2|S| + (d-1)/2},
\]
where the last inequality is due to the fact that $\dim_{\Phi}(1, v) = {{2|S| + v - 1}\choose {2|S| - 1}} \ll v^{2|S|}$. 

If we fix $d$ to be a large enough constant so that $d > 2|S| + (d-1)/2$, then for $v$ large enough we have
$\dim_{\PL}(d, v) > \dim_{\Phi}(d, v)$.
\end{proof}

\begin{proposition}\label{prop:dim}   
Suppose $\vert S\vert =2$ and $d<6$. Then for all $v \ge 1$, we have
\[\dim_{\PL}(d,v)\leq \dim_{\Phi}(d,v).\]
\end{proposition}

\begin{proof}
Note that 
\[
\dim_{\PL}(5,v) = \sum_{k=0}^{[v/5]}\dim_{\PL}(4,v-5k) \qquad \text{ and } \qquad \dim_{\Phi}(5,v) = \sum_{k=0}^{[v/5]}\dim_{\Phi}(3,v-5k).
\]
Hence, it suffices to prove that for every $v \ge 0$, one has:
\[
\dim_{\PL}(4,v) \le \dim_{\Phi}(3,v).
\]
On the one hand, $\dim_{\PL}(4,v)$ counts the number of ways to express $v$ as a sum of distinct weights of the form $(1,1,2,3,4)$, with repetitions and without order, while on the other hand $\dim_{\Phi}(3,v)$ counts the number of ways to express $v$ as a sum of distinct weights of the form $(1,1,1,1,3)$, under similar constraints. A $\PL$-partition of $v$ is a vector $(a,b,c,d,e)$ of non-negative integers, having the property that $a + b + 2c + 3d + 4e = v$. One way to prove the inequality $\dim_{\PL}(4,v) \le \dim_{\Phi}(3,v)$ is to produce an injection $(a,b,c,d,e)\mapsto (a',b',c',d',e')$ from the set of $\PL$-partitions of an arbitrary fixed $v$ in depth $4$ to the set of $\Phi$-partitions of the same $v$ in depth $3$. An example of such an injection is given by mapping a vector $(a,b,c,d,e)$ to $(a + \Delta(a,b,c,d,e), b,c,d,e)$, where
\[
\Delta(a,b,c,d,e) := v - a - b - c - d - 3e.
\]
\end{proof}

When $|S| = 1$, we already obtain the inequality 
\begin{equation}\label{eq:polylog-inequality}
   \dim_{\PL}(d, v) > \dim_\Phi(d, v) 
\end{equation}
when $d = 2$ and $v = 2$. This phenomenon explains the relative ease of solving the geometric step when $|S| = 1$ (see \Cref{sec:s=1}). When $|S| = 2$, the following table exhibits the first degree $v$ at which the inequality \eqref{eq:polylog-inequality} becomes true in depth $d$ for $1\leq d\leq 30$:

\begin{centering}
\begin{table}[H]
\begin{tabular}{| c | c | c | c |}
\hline
$d$ & $v$ & $\dim_{\Phi}(d,v)$ & $\dim_{\PL}(d,v)$ \\
\hline
1 & - & - & - \\
2 & - & - & - \\
3 & - & - & - \\
4 & - & - & - \\
5 & - & - & - \\
6 & 251 & 622565228 & 622894943 \\
7 & 291 & 9727962025 & 9751434234 \\
8 & 99 & 21381332 & 21582623 \\
9 & 109 & 87699272 & 87913253 \\
10 & 76 & 9681421 & 9802462 \\
11 & 82 & 25152148 & 25606281 \\
12 & 68 & 7495018 & 7506398 \\
13 & 72 & 14679671 & 14817938 \\
14 & 65 & 7354311 & 7370562 \\
15 & 68 & 12174636 & 12339732 \\
\hline

\end{tabular}
\qquad \qquad
\begin{tabular}{| c | c | c | c |}
\hline
$d$ & $v$ & $\dim_{\Phi}(d,v)$ & $\dim_{\PL}(d,v)$ \\
\hline
16 & 64 & 7960970 & 8045514 \\ 
17 & 66 & 11301646 & 11463717 \\
18 & 64 & 9050983 & 9286340 \\
19 & 65 & 11108926 & 11275641 \\
20 & 63 & 8824385 & 8838834 \\
21 & 64 & 10558940 & 10574205 \\
22 & 63 & 9384203 & 9394631 \\
23 & 64 & 11044181 & 11134313 \\
24 & 64 & 11044181 & 11347166 \\
25 & 64 & 11399096 & 11523873 \\
26 & 64 & 11399096 & 11670040 \\
27 & 64 & 11654983 & 11790526 \\
28 & 64 & 11654983 & 11889539 \\
29 & 64 & 11837155 & 11970650 \\
30 & 64 & 11837155 & 12036909 \\
\hline
\end{tabular}
\caption{The first degree $v$ for which  $\dim_{\PL}(d, v) > \dim_\Phi(d, v)$ when $|S|=2$.}\label{tab:s=1}
\end{table}
\end{centering}

This entries of this table were calculated using SageMath \cite{sagemath}. The code is available at our \href{https://github.com/Ariel-Z-Weiss/polylog-chabauty-kim}{Github repository}.

The blank rows in depths $d=1,2,3,4,5$ for $\vert S\vert =2$ indicate that the inequality $\dim_{\PL}(d, v) > \dim_\Phi(d, v)$ is never satisfied, in agreement with \Cref{prop:dim}. Note that the combination of \Cref{clever} with the row $d=6$ of \Cref{tab:s=1} implies the existence of a motivic Chabauty--Kim function of depth $6$ and weight $251$.

\subsection{A naive approach: computing the kernel of $\theta^{\#}_{d, v}$}

The matrix $M(\theta^{\#}_{d, v})$ has size $\dim_\Phi(d, v)\times \dim_{\PL}(d, v)$ with coefficients in $\oh(U_S)_{\leq d}$ and by \Cref{prop:dpl-dphi}, once $d$ and $v$ are large enough, it is guaranteed to have a non-trivial kernel. Hence, one could attempt to find Chabauty--Kim functions by explicitly computing these matrices.

\subsubsection{The case $|S|= 1$}\label{sec:s=1}

In the case $\vert S\vert=1$, this approach works well. We are guaranteed to find a Chabauty--Kim function already in depth $2$ and degree $2$. In this case, one can write down the corresponding size $3 \times 4$ matrix $M(\theta_{d, v}^\#)$ by hand and compute the kernel. 

For the sake of illustration, we perform this calculation. Write $\tau\in \Sigma_{-1}$ and $\sigma\in \Sigma_{-3}$. Then the map $\theta^{\#}$ in depth $2$ is given by:
\begin{align*}
    \theta^{\#}(\log^{\fu}) & = f_{\tau}\Phi_{e_0}^{\tau} \\
    \theta^{\#}(\Li_1^{\fu}) & = f_{\tau}\Phi_{e_1}^{\tau} \\
    \theta^{\#}(\Li_2^{\fu}) & = \frac{1}{2}f_{\tau}^2\Phi_{e_0}^{\tau}\Phi_{e_1}^{\tau}.
\end{align*}
Bases for the domain and codomain of $\theta^\#_{2, 2}$ are given by 
\[
\{ (\log^{\fu})^2, \log^{\fu}\Li_1^{\fu}, (\Li_1^{\fu})^2, \Li_2^{\fu} \} \quad \text{ and } \quad \{ (\Phi_{e_0}^{\tau})^2, (\Phi_{e_1}^{\tau})^2, \Phi_{e_0}^{\tau}\Phi_{e_1}^{\tau}  \}.
\]
With respect to these bases, we obtain the matrix 
\[
M(\theta_{2,2}^\#)=\left( \begin{matrix} f_\tau^2 & 0 & 0 & 0 \\ 0 & 0 & f_\tau^2 & 0 \\ 0 & f_\tau^2 & 0 & f_\tau^2/2 \end{matrix} \right).
\]
The kernel of this matrix is generated by the function \eqref{F122} 
\begin{equation*}
F^{\vert 1\vert}_{2,2}=\Li_2^{\fu}-\frac{1}{2}\log^{\fu}\Li_1^{\fu},
\end{equation*}
which belongs to $\mathcal{I}_{S, 2}^{\PL}$ when $\vert S\vert =1$. This is the polylogarithmic motivic Chabauty--Kim function of depth $2$ and degree $2$ constructed and studied in \cite{DCW2}. In \cite{CDC1}, the geometric step for $\vert S\vert =1$ is carried out in depth $4$ and degree $4$ by hand, resulting in a depth $4$ polylogarithmic motivic Chabauty--Kim function
\[
F^{\vert 1 \vert}_{4,4}=f_\sigma f_\tau \Li_4^{\fu}-f_{\sigma\tau}\log^{\fu}\Li_3^{\fu}-\frac{(\log^{\fu})^3 \Li_1^{\fu}}{24}(f_\sigma f_\tau - 4f_{\sigma\tau})\in \mathcal{I}_{S, 4}^{\PL}.
\]

\subsubsection{The case $|S|\ge 2$}

One could attempt a similar strategy in the general case, however, it quickly proves computationally infeasible. The matrix $M(\theta^{\#}_{d, v})$ has coefficients in the shuffle algebra $\oh(U_S)_{\leq d}$. As an abstract algebra, the
shuffle algebra $\oh(U_S)$ is isomorphic to the free commutative algebra $\Q[\mathcal{L}]$ generated by the set $\mathcal{L}$ of \emph{Lyndon words} of $\Sigma$ over $\Q$~\cite[\S 4 Thm. (i)]{LYNDON}. We may thus view $M(\theta^{\#}_{d, v})$ as a matrix with coefficients in a polynomial ring over $\Q$.

When $\vert S\vert=2$, by \Cref{tab:s=1}, the minimal matrix size at which \eqref{eq:polylog-inequality} holds is for $(d, v)=(14, 65)$. When $d = 14$, $296$ distinct Lyndon words show up in the image of $\theta^{\#}_d$. Thus, this brute force approach requires computing the kernel of a matrix of size $7354311 \times 7370562$ over the ring $\Q[X_1, \ldots, X_{296}]$, which is infeasible.

A key contribution of this paper, which we outline in the next section, is a method to produce polylogarithmic motivic Chabauty--Kim functions without computing this kernel using matrix operations. In the remainder of this section, we show how the naive method can nevertheless be salvaged to give an upper bound on the dimension of the kernel of $M(\theta^{\#}_{d, v})$.

\subsection{Non-existence of Chabauty--Kim functions in low degree}

Recall that $M(\theta^{\#}_{d, v})$ is a matrix over the shuffle algebra $\oh(U_S)_{\leq d}$, which we can view as a polynomial ring generated by the Lyndon words $\mathcal{L}_{\geq -d}$ of degree $\geq -d$ on the generators $\Sigma$. Write $\oh(U_S)_{\leq d} = \Q[\mathcal{L}_{\geq -d}]= \Q[X_1, \ldots, X_n]$, where $n$ is the number of Lyndon words on the set $\Sigma$ of degree $\geq -d$.

Let $x \in \Z^n$. Then, we can evaluate $M(\theta^{\#}_{d, v})$ at $x$ to obtain a matrix $M(\theta^{\#}_{d, v})(x)$ of the same size as $M(\theta^{\#}_{d, v})$, but with coefficients in $\Q$. Observe that if $f\in \ker (\theta^{\#}_{d, v})\sub \oh(U_S)_{\leq d}[\log^\uu, \Li_1^\uu, \ldots, \Li_d^\uu]$, then its evaluation $f(x) \in \Q[\log^\uu, \Li_1^\uu, \ldots, \Li_d^\uu]$ is in the kernel of $M(\theta^{\#}_{d, v})(x)$.

\begin{proposition}\label{prop:guestimate}
    For all $x\in \Z^n$, we have
    \[\mathrm{rank}_{\oh(U_S)_{\leq d}} \ker M(\theta^{\#}_{d, v})\le \dim_\Q\ker M(\theta^{\#}_{d, v})(x).\]
\end{proposition}

\begin{proof}
    The rank of the kernel of a matrix with coordinates in the shuffle algebra can only increase upon specialisation of the matrix to a $\Q$-point of $\spec \oh(U_S)_{\leq d}$. Indeed, any non-trivial linear dependence between rows before substitution persists after substitution.
\end{proof}

\begin{remark}\label{rem:Zar}
    Conversely, for a general point $x\in \Z^n$, we have
    \[\mathrm{rank}_{\oh(U_S)_{\leq d}} \ker M(\theta^{\#}_{d, v}) = \dim_\Q\ker M(\theta^{\#}_{d, v})(x).\]
    Indeed, in order for the inequality $\mathrm{rank}_{\oh(U_S)_{\leq d}} \ker M(\theta^{\#}_{d, v})< \dim_\Q\ker M(\theta^{\#}_{d, v})(x)$ to hold, the rank $r$ of $M(\theta^{\#}_{d, v})$ must be greater than the rank of $M(\theta^{\#}_{d, v})(x)$. The point $x$ must therefore be a root of all the non-trivial $r\times r$ minors of $M(\theta^{\#}_{d, v})$. Thus, $x$ must lie in the intersection of finitely many hypersurfaces in $\A^n$, so in particular, in a Zariski closed subset.
\end{remark}

The utility of the observations made in Proposition \ref{prop:guestimate} and Remark \ref{rem:Zar} is that, while computing the kernel of $M(\theta^{\#}_{d, v})$ is computationally hard, computing the dimension of the kernel of $M(\theta^{\#}_{d, v})(x)$ for any $x\iZ^n$ is computationally easy.

Motivated by \Cref{prop:guestimate}, we have implemented the following algorithm in SageMath \cite{sagemath}. It takes as input a tuple of integers $(s, d, v)$, where $s = |S|$, and outputs an integer $r$, which is the dimension of $\ker (M(\theta^{\#}_{d, v})(x))$ for a random $x\iZ^n$. By \Cref{prop:guestimate}, if the output is $0$, then there are no Chabauty--Kim functions in depth $d$ and degree $v$ for $|S|=s$. Moreover, if the output is $r$ for several iterations of the same tuple $(s, d,v)$, then it is likely that the dimension of $M(\theta^{\#}_{d, v})(x)$ is exactly $r$.

\begin{algorithm}[Upper Bound Algorithm]\label{algo}
    \hfill
    
    \textbf{Input: }a triple of integers $(s,d,v)$, where $s$ is the number of primes in $S$, $d$ is the depth, and $v$ is the degree.

    \textbf{Process:}
    \begin{enumerate}
        \item Compute all monomials of degree $v$ in the variables $\log^\uu, \Li_1^\uu, \ldots, \Li_d^{\uu}$. These monomials form a basis of $(\oh(U_S)_{\leq d}[\log^\uu, \Li_1^\uu, \ldots, \Li_d^\uu])_v$ as an $\oh(U_S)_{\leq d}$-module.
        \item Compute all monomials of degree $v$ in the variables $\Phi_{\lambda}^\rho$ for $\rho\in \{\tau_\l, \sigma_{2k+1}\}$ and $\lambda = e_0, e_1, e_1e_0^{2k}$. These monomials form a basis of $(R_{\Phi, d})_v$ as an $\oh(U_S)_{\leq d}$-module.
        \item Compute the set of Lyndon words $\mathcal{L}_{\geq -d}$ of $\Sigma$ to express $\oh(U_S)_{\leq d}$ as a polynomial algebra $\Q[\mathcal{L}_{\geq -d}]=\Q[X_1, \ldots, X_n]$.
        \item Compute $\theta_{d,v}^\#(\log^\uu)$, $\theta_{d,v}^\#(\Li_1^\uu),\ldots, \theta_{d,v}^\#(\Li_d^\uu)$ as elements of $\Q[X_1, \ldots, X_n]$. \footnote{This step is the bottleneck of the algorithm, since it requires converting all words of length $d$ in $2$ letters into the dual PBWL basis.}
        \item Choose a random $n$-tuple $x\iZ^n$ and evaluate each of the polynomials $\theta_{d,v}^\#(\log^\uu))$, $\theta_{d,v}^\#(\Li_i^\uu)$ at $x$.
        \item Build the matrix $M(\theta_{d,v}^\#)(x)$ with coefficients in $\Q$.
        \item Compute the dimension $r$ of the kernel of this matrix, and return $r$.
    \end{enumerate}
    \textbf{Output: }an integer $r$ representing an estimate, and provably an upper bound, for the rank of the kernel of $M(\theta_{d,v}^\#)$.
\end{algorithm}

We have implemented this algorithm in SageMath \cite{sagemath}, and our code is available at our \href{https://github.com/Ariel-Z-Weiss/polylog-chabauty-kim}{Github repository}. Using the algorithm, we obtain the following result:

\begin{theorem}\label{thm:no}
    Let $|S| = 2$. There are no polylogarithmic motivic Chabauty--Kim functions for $(d, v)$ with $v < 18$.
\end{theorem}

\begin{proof}
    We ran Algorithm \ref{algo} with $(s,d,v)=(2,16,17)$, and found the output to be $r = 0$. We note that this is sufficient to show that there are no motivic Chabauty--Kim functions for $(d,v)=(17, 17)$, since the polynomial $\Li_{17}^\uu$ cannot appear in a motivic Chabauty--Kim function of degree $17$: the cocycle $\Phi^{\sigma_{17}}_{e_1e_0^{16}}$ appears in $\theta^{\sharp}(\Li_{17}^\uu)$ but not in $\theta^{\sharp}(\Li_{n}^\uu)$ for $n < 17$. 
\end{proof}

    The computation was carried out on The Ohio State University's high performance cluster, and took around $66$ hours across $96$ cores, using $432$GB of memory. The overwhelming majority of this time was spent converting words in the shuffle algebra $\O(U_S)$ into the dual PBWL basis (Poincar\'e--Birkhoff--Witt basis constructed on the Lyndon basis) using an improved version of SageMath's built-in \verb|ShuffleAlgebra.to_dual_pbw_element| method, rewritten to enable parallelisation, and to fix inefficiencies. 

\begin{remark}\label{rem:infeasible}
    The algorithm produces an output of $1$ when $(s, d, v) = (2, 6, 18)$, strongly suggesting, according to Remark \ref{rem:Zar}, that there should exist a Chabauty--Kim function in this depth and degree.
    For $(s, d, v)=(2, 6, 18)$, we have $\dim_{\PL}(d, v) = 996$ and $\dim_{\Phi}(d, v) = 4183$. Additionally, $30$ Lyndon words of degree $\geq -18$ appear in the image of $\theta^{\#}$.
    It appears computationally infeasible to directly compute the kernel of $M(\theta_{6,18}^\#)$. A heuristic argument supporting this observation, assuming Gaussian elimination for direct computation, suggests that the space complexity required to represent the entries of $M(\theta_{6,18}^\#)$ grows exponentially with each row reduction step. Indeed, suppose that, in some extreme best case scenario, each row of the matrix is made up of zeros and of polynomials of degree $2$ in a \textit{single} variable. Then, since each row reduction step requires multiplying every subsequent row by the current pivot, the degree of the $n\th$ pivot will be $2^{2^n}$. Since a polynomial of degree $N$ in one variable is represented by a $\Q$-vector of length $N+1$, we obtain a space complexity lower bound of $\Omega(2^{2^{\min(\dim_{\PL}(d, v), \dim_{\Phi}(d, v))}})$. Since time complexity is lower bounded by space complexity, our rough computational complexity estimate is $2^{2^{996}}$. To contextualise this astronomically large number, the estimated number of atoms in the observable universe is between $10^{80}$ and $10^{85}$, approximately $2^{256}$. 
\end{remark}

\begin{remark}
    Originally, it was not clear to us that a motivic Chabauty--Kim function of depth $6$ and weight $18$ existed. Running the algorithm with $(2,6,18)$ made us suspect the existence of such a function. The elimination theory performed in \Cref{chapter:resultants} gives a function of depth $6$ and degree $26$, which is a multiple of $(\log^{\mathrm{u}})^6$, and thus has a factor of degree $20$. The fact that the algorithm with $(2,6,18)$ returned $1$ made us look for an additional factor of degree $2$, which turned out to be $F^{\vert 1\vert}_{2,2}$ defined by \eqref{F122}. In this sense, Algorithm \ref{algo} guided our search. \Cref{rem:infeasible} explains and quantifies why a brute force approach to finding a function of depth $6$ and degree $18$ is infeasible (contrary to the case $\vert S\vert=1$). Finally, we note that the $(2,6,18)$ calculation runs fairly quickly on a home computer ($10$ seconds), whereas the $(2,16,17)$ calculation took $66$ hours on a cluster.
\end{remark}

\section{The geometric step via resultants}
\label{chapter:resultants}

Inspired by ideas in \cite{DCJ}, we introduce a method using resultants, whose goal is to produce a polylogarithmic motivic Chabauty--Kim function of depth $\leq d$, i.e., a non-zero element in the kernel of the homomorphism (\Cref{theta})
\[\theta^{\#}_d: \oh(U_S)_{\leq d}[\log^{\fu},\Li^{\fu}_1,...,\Li^{\fu}_{d}]\to \oh(U_S)_{\leq d}[\Phi, d]. \]
We assume from now on that $|S|=2$ and write $S=\{ p, q \}$, so that $\Sigma_{-1} = \{\tau_{p},\tau_{q}\}$. With this assumption, the output of this section is an explicit non-trivial element in $\ker(\theta^{\sharp}_{6,18})$ (Definition \ref{def:F618} and Theorem \ref{prop:F5}). 

\subsection{The equations defining the map $\theta^{\sharp}$}\label{s:equations}

Making \Cref{theta} explicit in the case $\vert S\vert=2$ gives the formulas:
\begin{align}
\label{eq:theta log for two primes}
\theta^{\#}(\log^{\fu})  & = f_{\tau_{p}} \Phi_{e_{0}}^{\tau_{p}} +  f_{\tau_{q}} \Phi_{e_{0}}^{\tau_{q}} \\ 
\label{eq:theta Li1}
\theta^{\#}(\Li_{1}^{\fu}) & = f_{\tau_{p}} \Phi_{e_{1}}^{\tau_{p}} +  f_{\tau_{q}} \Phi_{e_{1}}^{\tau_{q}} \\ 
\label{eq:theta Li2}
\theta^{\#}(\Li^{\fu}_{2}) & =  f_{\tau_{p}\tau_{p}} \Phi_{e_{1}}^{\tau_{p}} \Phi_{e_{0}}^{\tau_{p}} +  f_{\tau_{p}\tau_{q}} \Phi_{e_{1}}^{\tau_{p}} \Phi_{e_{0}}^{\tau_{q}} +  f_{\tau_{q}\tau_{p}} \Phi_{e_{1}}^{\tau_{q}} \Phi_{e_{0}}^{\tau_{p}} + f_{\tau_{q}\tau_{q}} \Phi_{e_{1}}^{\tau_{q}} \Phi_{e_{0}}^{\tau_{q}} \\ 
\label{eq:theta Li3}
\theta^{\#}(\Li^{\fu}_{3}) & = f_{\sigma_3} \Phi_{e_{1}e_{0}^{2}}^{\sigma_{3}} + \sum_{\tau_{(1)},\tau_{(2)},\tau_{(3)} \in \{\tau_{p},\tau_{q}\}} f_{\tau_{(1)}\tau_{(2)}\tau_{(3)} } \Phi_{e_{1}}^{\tau_{(1)}} \Phi_{e_{0}}^{\tau_{(2)}} \Phi_{e_{0}}^{\tau_{(3)}} \\  
\label{eq:theta Li4}
\theta^{\#}(\Li^{\fu}_{4}) & = \sum_{\tau\in \{\tau_{p},\tau_{q}\}} f_{\sigma_3 \tau} \Phi_{e_{1}e_{0}^{2}}^{\sigma_{3}} \Phi_{e_{0}}^{\tau} + \sum_{\tau_{(1)},\tau_{(2)},\tau_{(3)},\tau_{(4)} \in \{\tau_{p},\tau_{q}\}} f_{\tau_{(1)}\tau_{(2)}\tau_{(3)}\tau_{(4)}} \Phi_{e_{1}}^{\tau_{(1)}} \Phi_{e_{0}}^{\tau_{(2)}} \Phi_{e_{0}}^{\tau_{(3)}}  \Phi_{e_{0}}^{\tau_{(4)}}.
\end{align}
More generally, for $n\geq 2$, we have
\begin{multline}  \label{eq:theta Lin for two primes odd}
\theta^{\#}(\Li^{\fu}_{2n-1}) = f_{\sigma_{2n-1}}\Phi_{e_1e_0^{2n-2}}^{\sigma_{2n-1}}+ 
\sum_{\substack{ 3\leq r\text{ odd} \leq 2n-2 \\ \tau_{(r+1)}, \ldots, \tau_{(2n-1)} \in \{\tau_{p},\tau_{q}\}}} f_{\sigma_{r} \tau_{(r+1)} \dots \tau_{(2n-1)}}
 \Phi_{e_{1}e_{0}^{r-1}}^{\sigma_{r}} \Phi_{e_{0}}^{\tau_{(r+1)}} \ldots 
 \Phi_{e_{0}}^{\tau_{(2n-1)}} \\
+\sum_{\tau_{(1)} ,\ldots, \tau_{(2n-1)}\in \{\tau_{p},\tau_{q}\}} f_{\tau_{(1)} \dots \tau_{(2n-1)}} \Phi_{e_{1}}^{\tau_{(1)}}\Phi_{e_{0}}^{\tau_{(2)}} \ldots \Phi_{e_{0}}^{\tau_{(2n-1)}}, 
\end{multline}
\begin{multline}  \label{eq:theta Lin for two primes pair}
\theta^{\#}(\Li^{\fu}_{2n}) = \sum_{\tau\in \{ \tau_p, \tau_q  \}} f_{\sigma_{2n-1}\tau}\Phi_{e_1e_0^{2n-2}}^{\sigma_{2n-1}}\Phi_{e_0}^{\tau}+ 
\sum_{\substack{ 3\leq r\text{ odd} \leq 2n-2 \\ \tau_{(r+1)}, \ldots, \tau_{(2n)} \in \{\tau_{p},\tau_{q}\}}} f_{\sigma_{r} \tau_{(r+1)} \dots \tau_{(2n)}}
 \Phi_{e_{1}e_{0}^{r-1}}^{\sigma_{r}} \Phi_{e_{0}}^{\tau_{(r+1)}} \ldots 
 \Phi_{e_{0}}^{\tau_{(2n)}} \\
+\sum_{\tau_{(1)} ,\ldots, \tau_{(2n)}\in \{\tau_{p},\tau_{q}\}} f_{\tau_{(1)} \dots \tau_{(2n)}} \Phi_{e_{1}}^{\tau_{(1)}}\Phi_{e_{0}}^{\tau_{(2)}} \ldots \Phi_{e_{0}}^{\tau_{(2n)}}.
\end{multline}

\subsection{The elimination process}\label{s:elim}

We now describe the process of ``eliminating'' variables $\Phi_{\lambda}^{\rho}$, i.e., expressing them in terms of other variables $\Phi_{\lambda'}^{\rho'}$ and images under $\theta^{\#}$ in order to obtain a Chabauty--Kim function. 
The system of equations described in \Cref{s:equations} is linear with respect to the variables $\Phi^{\sigma_{2n-1}}_{e_{1}e_{0}^{2n-2}}$ and the variables $\Phi_{e_{1}}^{\tau_p}$ and $\Phi_{e_{1}}^{\tau_q}$. Consequently, we first seek to eliminate these variables. Eventually, the elimination of the variables $\Phi^{\tau_p}_{e_0}$ and $\Phi^{\tau_q}_{e_0}$ requires the use of a resultant. Throughout this section, we will be working with polynomials with coefficients in the fraction field $K=\mathrm{Frac}(\oh(U_S))$.

\subsubsection{Elimination of $\Phi^{\sigma_{3}}_{e_{1}e_{0}^{2}}$ and $\Phi^{\sigma_{5}}_{e_{1}e_{0}^{4}}$} \label{s:elimsigma}

Using equation \eqref{eq:theta Lin for two primes odd} with $n\in \{2, 3\}$, we see that
\begin{equation}  \label{eq:elimination Phi_sigma_3 for two primes}
\Phi^{\sigma_{3}}_{e_{1}e_{0}^{2}} = \frac{1}{f_{\sigma_{3}}} \bigg( 
\theta^{\#}(\Li^{\fu}_{3}) -  
\sum_{\tau_{(1)} ,\tau_{(2)}, \tau_{(3)}} f_{\tau_{(1)} \tau_{(2)} \tau_{(3)}} \Phi_{e_{1}}^{\tau_{(1)}}\Phi_{e_{0}}^{\tau_{(2)}} \Phi_{e_{0}}^{\tau_{(3)}} \bigg),
\end{equation}
\begin{equation}  \label{eq:elimination Phi_sigma_5 for two primes}
\Phi^{\sigma_{5}}_{e_{1}e_{0}^{4}} = \frac{1}{f_{\sigma_{5}}} \bigg( 
\theta^{\#}(\Li^{\fu}_{5}) -  
\sum_{\tau_{(4)}, \tau_{(5)}} f_{\sigma_{3} \tau_{(4)} \tau_{(5)}}
 \Phi_{e_{1}e_{0}^{2}}^{\sigma_{3}} \Phi_{e_{0}}^{\tau_{(4)}} 
 \Phi_{e_{0}}^{\tau_{(5)}} - \sum_{\tau_{(1)}, \ldots, \tau_{(5)}} f_{\tau_{(1)}\tau_{(2)}\tau_{(3)}\tau_{(4)}\tau_{(5)}} \Phi_{e_{1}}^{\tau_{(1)}}\Phi_{e_{0}}^{\tau_{(2)}} \ldots \Phi_{e_{0}}^{\tau_{(5)}} \bigg).
\end{equation}
By substituting \eqref{eq:elimination Phi_sigma_3 for two primes} in \eqref{eq:elimination Phi_sigma_5 for two primes}, we obtain
\begin{multline}  \label{eq:elimination Phi_sigma_5 for two primes'}
\Phi^{\sigma_{5}}_{e_{1}e_{0}^{4}} = 
\frac{1}{f_{\sigma_{5}}} \bigg( 
\theta^{\#}(\Li^{\fu}_{5}) -  
\sum_{\tau_{(4)} ,\tau_{(5)}} 
\frac{f_{\sigma_{3} \tau_{(4)} \tau_{(5)}}}{f_{\sigma_{3}}} \bigg( 
\theta^{\#}(\Li^{\fu}_{3}) -  
\sum_{\tau_{(1)} ,\tau_{(2)},\tau_{(3)}} f_{\tau_{(1)} \tau_{(2)} \tau_{(3)}} \Phi_{e_{1}}^{\tau_{(1)}}\Phi_{e_{0}}^{\tau_{(2)}} \Phi_{e_{0}}^{\tau_{(3)}} \bigg)
 \Phi_{e_{0}}^{\tau_{(4)}} 
 \Phi_{e_{0}}^{\tau_{(5)}}
 \\ - \sum_{\tau_{(1)},\ldots, \tau_{(5)}} f_{\tau_{(1)}\tau_{(2)}\tau_{(3)}\tau_{(4)} \tau_{(5)}} \Phi_{e_{1}}^{\tau_{(1)}}\Phi_{e_{0}}^{\tau_{(2)}} \Phi_{e_{0}}^{\tau_{(3)}}\Phi_{e_{0}}^{\tau_{(4)}} \Phi_{e_{0}}^{\tau_{(5)}} \bigg),
\end{multline}
which can be rewritten as
\begin{multline}  \label{eq:elimination Phi_sigma_5 for two primes' 2}
\Phi^{\sigma_{5}}_{e_{1}e_{0}^{4}} = 
\frac{1}{f_{\sigma_{5}}} 
\theta^{\#}(\Li^{\fu}_{5}) -  
\sum_{\tau_{(4)} ,\tau_{(5)}} 
\frac{f_{\sigma_{3} \tau_{(4)} \tau_{(5)}}}{f_{\sigma_{3}} f_{\sigma_5}} 
\theta^{\#}(\Li^{\fu}_{3}) \Phi_{e_{0}}^{\tau_{(4)}} 
 \Phi_{e_{0}}^{\tau_{(5)}} \\ 
 +  
\sum_{\tau_{(1)} ,\ldots, \tau_{(5)}} \frac{1}{f_{\sigma_{3}}f_{\sigma_5}} \bigg( f_{\tau_{(1)} \tau_{(2)} \tau_{(3)}} f_{\sigma_{3} \tau_{(4)} \tau_{(5)}} - f_{\tau_{(1)}\tau_{(2)}\tau_{(3)}\tau_{(4)} \tau_{(5)}} f_{\sigma_{3}}\bigg)\Phi_{e_{1}}^{\tau_{(1)}}\Phi_{e_{0}}^{\tau_{(2)}} \Phi_{e_{0}}^{\tau_{(3)}}
 \Phi_{e_{0}}^{\tau_{(4)}} 
 \Phi_{e_{0}}^{\tau_{(5)}}.
\end{multline}

\subsubsection{Elimination of $\Phi_{e_{1}}^{\tau_{p}}$ and $\Phi_{e_{1}}^{\tau_{q}}$} \label{s:elim1}

Equations \eqref{eq:theta Li1} and \eqref{eq:theta Li2} are equivalent to the linear system
\begin{equation} \label{eq:linear system}
\left( \begin{matrix} \theta^{\#}(\Li^{\fu}_1) \\ \theta^{\#}(\Li^{\fu}_2) \end{matrix} \right) = \left( \begin{matrix} f_{\tau_p} & f_{\tau_q} \\ f_{\tau_p^2}\Phi_{e_0}^{\tau_p}+f_{\tau_p\tau_q}\Phi_{e_0}^{\tau_q} & f_{\tau_q\tau_p}\Phi_{e_0}^{\tau_p}+f_{\tau_q^2}\Phi_{e_0}^{\tau_q}  \end{matrix}\right) \left( \begin{matrix} \Phi_{e_1}^{\tau_p} \\ \Phi_{e_1}^{\tau_q} \end{matrix} \right).
\end{equation}
Denoting by $\Delta$ the determinant of the $2\times 2$ matrix in \eqref{eq:linear system}, we have
\begin{equation} \label{eq: elimination Phi_e1 for two primes}
\Delta \left( \begin{matrix} \Phi_{e_1}^{\tau_p} \\ \Phi_{e_1}^{\tau_q} \end{matrix} \right)=\left( \begin{matrix} f_{\tau_q\tau_p}\Phi_{e_0}^{\tau_p}+f_{\tau_q^2}\Phi_{e_0}^{\tau_q} & -f_{\tau_q} \\ -f_{\tau_p^2}\Phi_{e_0}^{\tau_p}-f_{\tau_p\tau_q}\Phi_{e_0}^{\tau_q} &   
f_{\tau_p}
\end{matrix}\right) \left( \begin{matrix} \theta^{\#}(\Li^{\fu}_1) \\ \theta^{\#}(\Li^{\fu}_2) \end{matrix} \right).
\end{equation}

Let $f_{[\tau_q, \tau_p]}:=f_{\tau_q \tau_p}- f_{\tau_p \tau_q}$.

\begin{lemma}\label{eq:detM}
The determinant $\Delta$ is equal to
$\frac{1}{2} f_{[\tau_q, \tau_p]}\theta^{\#}(\log^{\fu})$.
\end{lemma}

\begin{proof}
    Using \eqref{eq:theta log for two primes} and the shuffle equation $f_{\tau_p}f_{\tau_q}=f_{\tau_p \tau_q}+f_{\tau_q \tau_p}$, we see that
\begin{align*}
\Delta&=f_{\tau_p}(f_{\tau_q\tau_p}\Phi_{e_0}^{\tau_p}+f_{\tau_q^2}\Phi_{e_0}^{\tau_q})-f_{\tau_q}(f_{\tau_p^2}\Phi_{e_0}^{\tau_p}+f_{\tau_p\tau_q}\Phi_{e_0}^{\tau_q}) 
\\
& = f_{\tau_p}\Phi_{e_0}^{\tau_p}\left(f_{\tau_q \tau_p}-\frac{1}{2}f_{\tau_q}f_{\tau_p}\right)+f_{\tau_q}\Phi_{e_0}^{\tau_q}\left( \frac{1}{2}f_{\tau_p}f_{\tau_q}- f_{\tau_p \tau_q}\right).
\\ & =f_{\tau_p}\Phi_{e_0}^{\tau_p}\left(f_{\tau_q \tau_p}-\frac{1}{2}f_{\tau_q}f_{\tau_p}\right)+(\theta^{\#}(\log^{\fu})-f_{\tau_p}\Phi_{e_0}^{\tau_p})\left( \frac{1}{2}f_{\tau_p}f_{\tau_q}- f_{\tau_p \tau_q}\right)
\\
&=f_{\tau_p}\Phi_{e_0}^{\tau_p}\left(f_{\tau_q \tau_p}-\frac{1}{2}f_{\tau_q}f_{\tau_p} - \frac{1}{2}f_{\tau_p}f_{\tau_q}+ f_{\tau_p \tau_q}\right)+\theta^{\#}(\log^{\fu})\left( \frac{1}{2}f_{\tau_p}f_{\tau_q}- f_{\tau_p \tau_q}\right)
\\
&= 0 + \theta^{\#}(\log^{\fu})\left( \frac{1}{2}f_{\tau_q\tau_p}- \frac{1}{2}f_{\tau_p \tau_q}\right).
\end{align*}
\end{proof}

\begin{lemma}\label{lem:1p}
    We have 
    \[
    \Delta \Phi_{e_1}^{\tau_p}=\frac{1}{2}f_{[\tau_q,\tau_p]}\theta^{\#}(\Li_1^{\fu})\Phi_{e_0}^{\tau_p}-f_{\tau_q} \theta^{\#}(F^{\vert 1\vert}_{2,2}),
    \]
    where $F^{\vert 1\vert}_{2,2}$ was defined in \eqref{F122}.
\end{lemma}

\begin{proof}
    Using \eqref{eq: elimination Phi_e1 for two primes}, (\ref{eq:theta log for two primes}), and the shuffle equation $f_{\tau_p}f_{\tau_q}=f_{\tau_p \tau_q}+f_{\tau_q \tau_p}$, we see that
    \begin{align*}
        \Delta \Phi_{e_1}^{\tau_p}&=\bigg(f_{\tau_q\tau_p}\Phi_{e_0}^{\tau_p}+f_{\tau_q^2}\Phi_{e_0}^{\tau_q}\bigg)\theta^{\#}(\Li_1^{\fu})-f_{\tau_q}\theta^{\#}(\Li_2^{\fu}) \\
        & = \bigg(f_{\tau_q\tau_p}\Phi_{e_0}^{\tau_p}+f_{\tau_q^2}\bigg(\frac{1}{f_{\tau_q}}(\theta^{\#}(\log^{\fu})-f_{\tau_p}\Phi_{e_0}^{\tau_p})\bigg)\bigg)\theta^{\#}(\Li_1^{\fu})-f_{\tau_q}\theta^{\#}(\Li_2^{\fu}) \\        &=\bigg(f_{\tau_q\tau_p}\Phi_{e_0}^{\tau_p}+\frac{f_{\tau_q}}{2}(\theta^{\#}(\log^{\fu})-f_{\tau_p}\Phi_{e_0}^{\tau_p})\bigg)\theta^{\#}(\Li_1^{\fu})-f_{\tau_q}\theta^{\#}(\Li_2^{\fu}) \\       &=\bigg(f_{\tau_q\tau_p}-\frac{f_{\tau_p}f_{\tau_q}}{2}\bigg)\theta^{\#}(\Li_1^{\fu})\Phi_{e_0}^{\tau_p}+\frac{f_{\tau_q}}{2}\theta^{\#}(\log^{\fu}\Li_1^{\fu})-f_{\tau_q}\theta^{\#}(\Li_2^{\fu}) \\
        &=\frac{1}{2}f_{[\tau_q,\tau_p]}\theta^{\#}(\Li_1^{\fu})\Phi_{e_0}^{\tau_p}-f_{\tau_q} \theta^{\#}\bigg( \Li_2^{\fu}-\frac{1}{2}\log^{\fu}\Li_1^{\fu} \bigg).
    \end{align*}
\end{proof}

\begin{remark}
    Since $F^{\vert 1\vert}_{2,2}$ is of degree $2<18$, we deduce by \Cref{thm:no} that $\theta^{\#}(F^{\vert 1\vert}_{2,2})\neq 0$ even though $F^{\vert 1\vert}_{2,2}\in \cI^{\PL}_{S',2}$ when $\vert S'\vert=1$. One can also compute this fact by hand.
\end{remark}

\begin{lemma}\label{lem:1q}
    We have 
    \[
    f_{\tau_q}\Delta \Phi_{e_1}^{\tau_q}=(f_{\tau_p^2\tau_q}-f_{\tau_q\tau_p^2})\theta^{\#}(\Li_1^{\fu})\Phi_{e_0}^{\tau_p}
    +\big( f_{\tau_p}f_{\tau_q}\theta^{\#}(\Li_2^{\fu})-f_{\tau_p\tau_q}\theta^{\#}(\log^{\fu}\Li_1^{\fu})\big).
    \]
\end{lemma}

\begin{proof}
    Using \eqref{eq: elimination Phi_e1 for two primes} and (\ref{eq:theta log for two primes}), we see that
    \begin{align*}
        \Delta \Phi_{e_1}^{\tau_q}&=-\bigg(f_{\tau_p^2}\Phi_{e_0}^{\tau_p}+f_{\tau_p\tau_q}\Phi_{e_0}^{\tau_q}\bigg)\theta^{\#}(\Li_1^{\fu})+f_{\tau_p}\theta^{\#}(\Li_2^{\fu}) \\
        &=-\bigg(f_{\tau_p^2}\Phi_{e_0}^{\tau_p}+\frac{f_{\tau_p\tau_q}}{f_{\tau_q}}(\theta^{\#}(\log^{\fu})-f_{\tau_p}\Phi_{e_0}^{\tau_p})\bigg)\theta^{\#}(\Li_1^{\fu})+f_{\tau_p}\theta^{\#}(\Li_2^{\fu}) \\
        &=-\bigg(\bigg(f_{\tau_p^2}-\frac{f_{\tau_p\tau_q}f_{\tau_p}}{f_{\tau_q}}\bigg)\Phi_{e_0}^{\tau_p}+\frac{f_{\tau_p\tau_q}}{f_{\tau_q}}\theta^{\#}(\log^{\fu})\bigg)\theta^{\#}(\Li_1^{\fu})+f_{\tau_p}\theta^{\#}(\Li_2^{\fu}) \\
        &=\frac{1}{f_{\tau_q}}(f_{\tau_p\tau_q}f_{\tau_p}-f_{\tau_p^2}f_{\tau_q})\theta^{\#}(\Li_1^{\fu})\Phi_{e_0}^{\tau_p}+\frac{1}{f_{\tau_q}}\bigg(f_{\tau_p}f_{\tau_q}\theta^{\#}(\Li_2^{\fu})-f_{\tau_p\tau_q}\theta^{\#}(\log^{\fu}\Li_1^{\fu})\bigg).
    \end{align*}
    The result follows from the shuffle equation
    \[
    f_{\tau_p\tau_q}f_{\tau_p}-f_{\tau_p^2}f_{\tau_q}=2f_{\tau_p^2\tau_q}+f_{\tau_p\tau_q\tau_p}-f_{\tau_p^2\tau_q}-f_{\tau_p\tau_q\tau_p}-f_{\tau_q\tau_p^2}=f_{\tau_p^2\tau_q}-f_{\tau_q\tau_p^2}.
    \]
\end{proof}

\subsubsection{Elimination of $\Phi_{e_{0}}^{\tau_{p}}$ and $\Phi_{e_{0}}^{\tau_{q}}$} \label{s:elim0}

At this stage, we will encounter non-linear polynomial equations and we will need a resultant. Let us consider equation (\ref{eq:theta Lin for two primes pair}) for $d=2,3$: 
\begin{equation} \label{eq:theta Li4 for two primes}
\theta^{\#}(\Li^{\fu}_{4}) = \sum_{\tau} f_{\sigma_3 \tau} \Phi_{e_{1}e_{0}^{2}}^{\sigma_{3}} \Phi_{e_{0}}^{\tau} + \sum_{\tau_{(1)},\ldots, \tau_{(4)}} f_{\tau_{(1)}\tau_{(2)}\tau_{(3)}\tau_{(4)}} \Phi_{e_{1}}^{\tau_{(1)}} \Phi_{e_{0}}^{\tau_{(2)}} \Phi_{e_{0}}^{\tau_{(3)}}  \Phi_{e_{0}}^{\tau_{(4)}}, 
\end{equation}
\begin{equation} \label{eq:theta Li6 for two primes}
\theta^{\#}(\Li^{\fu}_{6}) =  \sum_{\tau} f_{\sigma_5 \tau} \Phi_{e_{1}e_{0}^{5}}^{\sigma_{5}} \Phi_{e_{0}}^{\tau}
+ 
 \sum_{\tau} f_{\sigma_3 \tau_{(4)}\tau_{(5)}\tau_{(6)}} \Phi_{e_{1}e_{0}^{2}}^{\sigma_{3}} \Phi_{e_{0}}^{\tau_{(4)}}\Phi_{e_{0}}^{\tau_{(5)}}\Phi_{e_{0}}^{\tau_{(6)}}
 + \sum_{\tau_{(1)},\ldots,\tau_{(6)}} f_{\tau_{(1)}\ldots \tau_{(6)}} \Phi_{e_{1}}^{\tau_{(1)}} \Phi_{e_{0}}^{\tau_{(2)}} \ldots \Phi_{e_{0}}^{\tau_{(6)}}. 
\end{equation}

Using Sections \ref{s:elimsigma} and \ref{s:elim1}, we can eliminate the variables $\Phi^{\sigma_{3}}_{e_{1}e_{0}^{2}}$, $\Phi^{\sigma_{5}}_{e_{1}e_{0}^{4}}$, $\Phi^{\tau_{p}}_{e_{1}}$, and $\Phi^{\tau_{q}}_{e_{1}}$ in (\ref{eq:theta Li4 for two primes}) and (\ref{eq:theta Li6 for two primes}), and regard the latter equations as two polynomial equations in $(\Phi_{e_{0}}^{\tau_{p}},\Phi_{e_{0}}^{\tau_{q}})$, as follows. 
First, in view of using \eqref{eq: elimination Phi_e1 for two primes}, we multiply equations (\ref{eq:theta Li4 for two primes}) and (\ref{eq:theta Li6 for two primes}) by $\Delta$, to obtain 
\begin{equation} \label{eq:theta Li4 bis}
\Delta\theta^{\#}(\Li^{\fu}_{4}) = \sum_{\tau} f_{\sigma_3 \tau} \Delta\Phi_{e_{1}e_{0}^{2}}^{\sigma_{3}} \Phi_{e_{0}}^{\tau} + \sum_{\tau_{(1)},\ldots,\tau_{(4)}} f_{\tau_{(1)}\tau_{(2)}\tau_{(3)}\tau_{(4)}} \Delta\Phi_{e_{1}}^{\tau_{(1)}} \Phi_{e_{0}}^{\tau_{(2)}} \Phi_{e_{0}}^{\tau_{(3)}}  \Phi_{e_{0}}^{\tau_{(4)}}, 
\end{equation}
\begin{multline} \label{eq:theta Li6 bis}
\Delta\theta^{\#}(\Li^{\fu}_{6}) = \sum_{\tau} f_{\sigma_5 \tau} \Delta\Phi_{e_{1}e_{0}^{4}}^{\sigma_{5}} \Phi_{e_{0}}^{\tau}
+ 
 \sum_{\tau_{(4)}, \tau_{(5)}, \tau_{(6)}} f_{\sigma_3 \tau_{(4)}\tau_{(5)}\tau_{(6)}} \Delta \Phi_{e_{1}e_{0}^{2}}^{\sigma_{3}} \Phi_{e_{0}}^{\tau_{(4)}}\Phi_{e_{0}}^{\tau_{(5)}}\Phi_{e_{0}}^{\tau_{(6)}} \\
 + \sum_{\tau_{(1)},\ldots,\tau_{(6)}} f_{\tau_{(1)}\ldots \tau_{(6)}} \Delta \Phi_{e_{1}}^{\tau_{(1)}} \Phi_{e_{0}}^{\tau_{(2)}} \ldots \Phi_{e_{0}}^{\tau_{(6)}}. 
\end{multline}
Then, in (\ref{eq:theta Li4 bis}) and (\ref{eq:theta Li6 bis}) we insert 
\begin{itemize}
    \item equations (\ref{eq:elimination Phi_sigma_3 for two primes}) and (\ref{eq:elimination Phi_sigma_5 for two primes'}) to eliminate $\Phi_{e_{1}e_{0}^{2}}^{\sigma_{3}}$ and $\Phi_{e_{1}e_{0}^{4}}^{\sigma_{5}}$;
    \item the formulas of Lemmas \ref{lem:1p} and \ref{lem:1q} to eliminate $\Delta \Phi_{e_{1}}^{\tau_{p}}$ and $\Delta\Phi_{e_{1}}^{\tau_{q}}$;
    \item the formula of Lemma \ref{eq:detM} to eliminate $\Delta$.
\end{itemize}
Then, (\ref{eq:theta Li4 for two primes}) and (\ref{eq:theta Li6 for two primes}) become polynomial equations in $(\Phi_{e_{0}}^{\tau_{p}},\Phi_{e_{0}}^{\tau_{q}})$. Finally, use equation (\ref{eq:theta log for two primes}) in the form
\begin{equation} \label{eq:theta log for two primes'}
\Phi_{e_0}^{\tau_q}=\frac{1}{f_{\tau_q}}(\theta^{\#}(\log^{\fu})-f_{\tau_p}\Phi_{e_0}^{\tau_p})
\end{equation}
to eliminate $\Phi_{e_{0}}^{\tau_{q}}$. After clearing denominators (i.e., multiplying by $f_{\tau_q}^4 f_{\sigma_3}$ in the case of $\Li_4^{\fu}$ and by $f_{\tau_q}^6 f_{\sigma_3} f_{\sigma_5}$ in the case of $\Li_6^{\fu}$), we obtain two polynomial equations in one variable of the following form (see the next paragraph for precise definitions and computations):
\begin{equation}  \label{eq:common root}
P_{4}(\Phi_{e_{0}}^{\tau_{p}}) = P_{6}(\Phi_{e_{0}}^{\tau_{p}})  = 0 
\end{equation}
where 
\[
    P_{i}(X) \in (\oh(U_S)[\theta^{\#}(\log^{\fu}), (\theta^{\#}(\Li_n^{\fu}))_{1\leq n \leq i}])[X], \qquad i\in \{ 4,6 \}.
\]
In particular, (\ref{eq:common root}) implies that $P_{4}(X)$ and $P_{6}(X)$ have a common root, namely $\Phi_{e_0}^{\tau_p}$, whence 
\begin{equation} \label{eq: resultant vanishesP} 
\mathrm{Res}(P_{4}(X),P_{6}(X))= 0.  
\end{equation}
Moreover, $P_{4}(X)$ and $P_{6}(X)$ have natural expressions of the form
$$ P_i(X)=\theta^{\#}(\nu_i(X)), \qquad i\in \{ 4,6 \}, $$
where $\theta^{\#}$ is applied to the coefficients of $\nu_i(X)$, and
\[
\nu_{i}(X) \in (\oh(U_S)[\log^{\fu},(\Li_n^{\fu})_{1\leq n \leq i}])[X], \qquad i\in \{ 4,6 \}.
\]
\begin{proposition}\label{eq: resultant vanishes} 
The element $\mathrm{Res}(\nu_{4}(X),\nu_{6}(X)) \in \oh(U_S)[\log^{\fu}, \Li_1^{\fu}, \ldots, \Li_6^{\fu}]$ belongs to $\ker(\theta^{\sharp})$.
\end{proposition}
\begin{proof}
We have
    \begin{equation} 
    \theta^{\sharp}(\mathrm{Res}(\nu_{4}(X),\nu_{6}(X)))=
 \mathrm{Res}(\theta^{\sharp}(\nu_{4}(X)),\theta^{\sharp}(\nu_{6}(X))) = \Res(P_4(X), P_6(X)) = 0, 
\end{equation}
where we used \eqref{eq: resultant vanishesP} in the last equality.
\end{proof}

\subsection{Computation of the Chabauty--Kim function $\mathrm{Res}(\nu_{4}(X),\nu_{6}(X))$}

We now compute the polynomials $\nu_{4}(X)$ and $\nu_{6}(X)$ evoked in \Cref{s:elim0}, in order to obtain a more explicit expression for the Chabauty--Kim function $\mathrm{Res}(\nu_{4}(X),\nu_{6}(X))$. Such an expression will be instrumental in proving its non-triviality, and in computing its degree.

\subsubsection{An additional formula for the elimination of $\Phi_{e_{0}}^{\tau_{q}}$} 

We are going to encounter several expressions of the following form, where $F$ is a function with values in $\oh(U_S)$: 
$$ \sum_{\tau_{(j_1)},\ldots,\tau_{(j_r)} \in \{\tau_{p},\tau_{q}\}} F(\tau_{(j_1)},\ldots,\tau_{(j_r)}) \Phi_{e_{0}}^{\tau_{(j_1)}} \cdots \Phi_{e_{0}}^{\tau_{(j_r)}}. $$
Using equation (\ref{eq:theta log for two primes'}), any expression of this type can be rewritten as 
\begin{equation} \label{eq:function} 
\sum_{i=0}^{r} \sum_{\substack{I \subset \{j_1,\ldots,j_r\} \\ |I| = i}}  F(\tau_{(j_1)}^I,\ldots,\tau_{(j_r)}^I) (\Phi_{e_{0}}^{\tau_{p}})^{i} \bigg(  \frac{1}{ f_{\tau_{q}}} \big( \theta^{\#}(\log^{\fu}) - f_{\tau_{p}} \Phi_{e_{0}}^{\tau_{p}} \big) \bigg)^{r-i},
\end{equation}
where $\tau_{(j_k)}^I = \left\{ \begin{array}{ll} \tau_{p} \text{  if }j_k \in I
\\ \tau_{q}\text{ if }j_k \not\in I.
\end{array}\right.$

\subsubsection{The polynomial $\nu_{4}(X)$} 

We start with equation (\ref{eq:theta Li4 bis}), which is the formula for $\theta^{\#}(\Li^{\fu}_{4})$ multiplied by $\Delta$. We insert equation (\ref{eq:elimination Phi_sigma_3 for two primes}) to eliminate $\Phi_{e_{1}e_{0}^{2}}^{\sigma_{3}}$. Moreover, whenever $\Delta\Phi_{e_{1}}^{\tau}$ appears, we separate the terms in $\tau = \tau_{p}$ and $\tau = \tau_{q}$. We obtain
\begin{multline} \label{eq:theta Li4 bis bis}  \Delta \theta^{\#}(\Li^{\fu}_{4}) = \sum_{\tau}  \frac{f_{\sigma_3 \tau}}{f_{\sigma_{3}}}  \bigg( \Delta 
\theta^{\#}(\Li^{\fu}_{3}) -  
\sum_{\tau_{(3)}, \tau_{(4)}}
\bigg[ f_{\tau_{p}\tau_{(3)}\tau_{(4)}} \Delta \Phi_{e_{1}}^{\tau_{p}}
+  f_{\tau_{q}\tau_{(3)}\tau_{(4)}} \Delta \Phi_{e_{1}}^{\tau_{q}} \bigg]
\Phi_{e_{0}}^{\tau_{(3)}} \Phi_{e_{0}}^{\tau_{(4)}} \bigg) \Phi_{e_{0}}^{\tau}
\\ + \sum_{\tau_{(2)},\tau_{(3)},\tau_{(4)}} \bigg[ f_{\tau_{p}\tau_{(2)}\tau_{(3)}\tau_{(4)}} \Delta \Phi_{e_{1}}^{\tau_{p}}
+  f_{\tau_{q}\tau_{(2)}\tau_{(3)}\tau_{(4)}} \Delta \Phi_{e_{1}}^{\tau_{q}} \bigg]\Phi_{e_{0}}^{\tau_{(2)}} \Phi_{e_{0}}^{\tau_{(3)}}  \Phi_{e_{0}}^{\tau_{(4)}}. 
\end{multline}
By collecting the $\Delta \Phi_{e_{1}}^{\tau_{p}}$ and $\Delta \Phi_{e_{1}}^{\tau_{q}}$ terms, this can be rewritten as
\begin{multline} \label{eq:theta Li4 bis bis bis} 
\Delta \theta^{\#}(\Li^{\fu}_{4}) = \sum_{\tau} \frac{f_{\sigma_3 \tau}}{f_{\sigma_{3}}} \Delta 
\theta^{\#}(\Li^{\fu}_{3}) \Phi_{e_{0}}^{\tau}
\\ + \sum_{\tau_{(2)},\tau_{(3)},\tau_{(4)}} \frac{1}{f_{\sigma_3}} \bigg[ \bigg(f_{\tau_{p}\tau_{(2)}\tau_{(3)}\tau_{(4)}} f_{\sigma_3} - f_{\tau_p\tau_{(3)}\tau_{(4)}}f_{\sigma_3 \tau_{(2)}}\bigg) \Delta \Phi_{e_{1}}^{\tau_{p}} \bigg. \\
\bigg. +  \bigg( f_{\tau_{q}\tau_{(2)}\tau_{(3)}\tau_{(4)}}f_{\sigma_3} - f_{\tau_q\tau_{(3)}\tau_{(4)}}f_{\sigma_3 \tau_{(2)}}  \bigg) \Delta \Phi_{e_{1}}^{\tau_{q}} \bigg]\Phi_{e_{0}}^{\tau_{(2)}} \Phi_{e_{0}}^{\tau_{(3)}}  \Phi_{e_{0}}^{\tau_{(4)}}.
\end{multline}
We use Lemmas \ref{lem:1p} and \ref{lem:1q} in (\ref{eq:theta Li4 bis bis bis}) to eliminate $\Delta \Phi_{e_{1}}^{\tau_{p}}$ and $\Delta \Phi_{e_{1}}^{\tau_{q}}$, and use Lemma \ref{eq:detM} to eliminate $\Delta$. We obtain 
\begin{multline} \label{eq:one before last P4 bis}
\frac{1}{2} f_{[\tau_q, \tau_p]}\theta^{\#}(\log^{\fu}) \theta^{\#}(\Li^{\fu}_{4}) = \sum_{\tau} \frac{f_{\sigma_3 \tau}f_{[\tau_q, \tau_p]}}{2f_{\sigma_{3}}}\theta^{\#}(\log^{\fu}) 
\theta^{\#}(\Li^{\fu}_{3}) \Phi_{e_{0}}^{\tau}
\\ + \sum_{\tau_{(2)},\tau_{(3)},\tau_{(4)}} \frac{1}{f_{\sigma_3}} \bigg[ \bigg(f_{\tau_{p}\tau_{(2)}\tau_{(3)}\tau_{(4)}} f_{\sigma_3} - f_{\tau_p\tau_{(3)}\tau_{(4)}}f_{\sigma_3 \tau_{(2)}}\bigg) \bigg( \frac{1}{2}f_{[\tau_q,\tau_p]}\theta^{\#}(\Li_1^{\fu})\Phi_{e_0}^{\tau_p}-f_{\tau_q} \theta^{\#}(F^{\vert 1\vert}_{2,2}) \bigg) \bigg. \\
\bigg. +  \frac{1}{f_{\tau_q}} \bigg( f_{\tau_{q}\tau_{(2)}\tau_{(3)}\tau_{(4)}}f_{\sigma_3} - f_{\tau_q\tau_{(3)}\tau_{(4)}}f_{\sigma_3 \tau_{(2)}}  \bigg) \bigg( (f_{\tau_p^2\tau_q}-f_{\tau_q\tau_p^2})\theta^{\#}(\Li_1^{\fu})\Phi_{e_0}^{\tau_p}
    +\big( f_{\tau_p}f_{\tau_q}\theta^{\#}(\Li_2^{\fu})-f_{\tau_p\tau_q}\theta^{\#}(\log^{\fu}\Li_1^{\fu})\big) \bigg) \bigg] \\
    \cdot \Phi_{e_{0}}^{\tau_{(2)}} \Phi_{e_{0}}^{\tau_{(3)}}  \Phi_{e_{0}}^{\tau_{(4)}}.
\end{multline}

Substituting equations (\ref{eq:theta log for two primes'}) and (\ref{eq:function}) in equation (\ref{eq:one before last P4 bis}) to eliminate $\Phi_{e_{0}}^{\tau_{q}}$, and clearing denominators by multiplying by $f_{\tau_q}^4 f_{\sigma_3}$, we obtain an equation of the form $P_{4}(\Phi_{e_{0}}^{\tau_{p}})=0$ where $P_{4}(X)= \theta^{\sharp}(\nu_{4}(X))$ with
\begin{multline} \label{eq:nu4}
\nu_4(X)= \\ -\frac{1}{2} f_{\tau_q}^4 f_{\sigma_3} f_{[\tau_q, \tau_p]}\log^{\fu} \Li^{\fu}_{4} + \frac{1}{2}f_{\tau_q}^3 f_{\sigma_3 \tau_q}f_{[\tau_q, \tau_p]}(\log^{\fu})^2 
\Li^{\fu}_{3} +\frac{1}{2}f_{\tau_q}^3 f_{[\tau_q, \tau_p]} \big( f_{\sigma_3 \tau_p}f_{\tau_q} - f_{\sigma_3 \tau_q}f_{\tau_p}  \big)\log^{\fu} 
\Li^{\fu}_{3} X
\\ + \sum_{i=0}^3 \sum_{\substack{I\subset \{ 2,3,4\} \\ \vert I\vert=i}} f_{\tau_q}^{i+1} \bigg[ \bigg(f_{\tau_{p}\tau_{(2)}^I\tau_{(3)}^I\tau_{(4)}^I} f_{\sigma_3} - f_{\tau_p\tau_{(3)}^I\tau_{(4)}^I}f_{\sigma_3 \tau_{(2)}^I}\bigg) \bigg( \frac{1}{2}f_{[\tau_q,\tau_p]}\Li_1^{\fu}X-f_{\tau_q} F^{\vert 1\vert}_{2,2} \bigg) \bigg. \\
\bigg. +  \frac{1}{f_{\tau_q}}\bigg( f_{\tau_{q}\tau_{(2)}^I\tau_{(3)}^I\tau_{(4)}^I}f_{\sigma_3} - f_{\tau_q\tau_{(3)}^I\tau_{(4)}^I}f_{\sigma_3 \tau_{(2)}^I}  \bigg) \bigg( (f_{\tau_p^2\tau_q}-f_{\tau_q\tau_p^2})\Li_1^{\fu}X
    +\big( f_{\tau_p}f_{\tau_q}\Li_2^{\fu}-f_{\tau_p\tau_q}\log^{\fu}\Li_1^{\fu}\big) \bigg) \bigg] \\ 
    \cdot X^i \big( \log^{\fu}-f_{\tau_p}X \big)^{3-i}.
\end{multline}

\begin{lemma}\label{lem:4cst2}
    The polynomial $\nu_4(X)$ has degree $2$ and is a multiple of $\log^{\fu}$. 
\end{lemma}

\begin{proof}
    The statement about the degree was verified using SageMath \cite{sagemath} (see our  \href{https://github.com/Ariel-Z-Weiss/polylog-chabauty-kim}{GitHub repository} for the code). It is then clear from \eqref{eq:nu4} that $\nu_4(X)$ is a multiple of $\log^{\fu}$. 
\end{proof}

\subsubsection{The polynomial $\nu_6(X)$}

We start with equation (\ref{eq:theta Li6 bis}) which is the formula for $\theta^{\sharp}(\Li^{\fu}_{6})$ multiplied by $\Delta$. We insert (\ref{eq:elimination Phi_sigma_3 for two primes}) and (\ref{eq:elimination Phi_sigma_5 for two primes' 2}) to eliminate $\Phi_{e_{1}e_{0}^{2}}^{\sigma_{3}}$ and $\Phi_{e_{1}e_{0}^{2}}^{\sigma_{5}}$. We obtain
\begin{multline}
\Delta\theta^{\#}(\Li^{\fu}_{6}) = \sum_{\tau}  \frac{f_{\sigma_5 \tau}}{f_{\sigma_5}} \Delta
\theta^{\#}(\Li^{\fu}_{5}) \Phi_{e_0}^{\tau} \\
+ \sum_{\tau_{(4)}, \tau_{(5)}, \tau_{(6)}}\frac{1}{f_{\sigma_{3}} f_{\sigma_5}} \bigg( f_{\sigma_3 \tau_{(4)}\tau_{(5)}\tau_{(6)}} f_{\sigma_5} - f_{\sigma_{3} \tau_{(4)} \tau_{(5)}} f_{\sigma_5 \tau_{(6)}} \bigg)
\Delta\theta^{\#}(\Li^{\fu}_{3}) \Phi_{e_{0}}^{\tau_{(4)}}\Phi_{e_{0}}^{\tau_{(5)}}\Phi_{e_{0}}^{\tau_{(6)}} \\
 +  
\sum_{\tau_{(1)} ,\ldots, \tau_{(6)}} \frac{1}{f_{\sigma_{3}}f_{\sigma_5}} \bigg( f_{\tau_{(1)} \tau_{(2)} \tau_{(3)}} f_{\sigma_{3} \tau_{(4)} \tau_{(5)}} f_{\sigma_5 \tau_{(6)}} - f_{\tau_{(1)}\tau_{(2)}\tau_{(3)}\tau_{(4)} \tau_{(5)}} f_{\sigma_{3}} f_{\sigma_5 \tau_{(6)}} \bigg. \\
\bigg. - f_{\tau_{(1)} \tau_{(2)} \tau_{(3)}} f_{\sigma_3 \tau_{(4)}\tau_{(5)}\tau_{(6)}}f_{\sigma_5} + f_{\tau_{(1)} \tau_{(2)}\ldots \tau_{(6)}} f_{\sigma_3} f_{\sigma_5} \bigg)\Delta\Phi_{e_{1}}^{\tau_{(1)}}\Phi_{e_{0}}^{\tau_{(2)}} \Phi_{e_{0}}^{\tau_{(3)}}
 \Phi_{e_{0}}^{\tau_{(4)}} 
 \Phi_{e_{0}}^{\tau_{(5)}}\Phi_{e_0}^{\tau_{(6)}}.
\end{multline}
Splitting the sums indexed by $\tau$ and $\tau_{(1)}$ yields
\begin{multline}\label{eq:theta Li6 bis bis}
\Delta\theta^{\#}(\Li^{\fu}_{6}) = \frac{f_{\sigma_5 \tau_p}}{f_{\sigma_5}} \Delta
\theta^{\#}(\Li^{\fu}_{5}) \Phi_{e_0}^{\tau_p} + \frac{f_{\sigma_5 \tau_q}}{f_{\sigma_5}} \Delta
\theta^{\#}(\Li^{\fu}_{5}) \Phi_{e_0}^{\tau_q} \\
+ \sum_{\tau_{(4)}, \tau_{(5)}, \tau_{(6)}}\frac{1}{f_{\sigma_{3}} f_{\sigma_5}} \bigg( f_{\sigma_3 \tau_{(4)}\tau_{(5)}\tau_{(6)}} f_{\sigma_5} - f_{\sigma_{3} \tau_{(4)} \tau_{(5)}} f_{\sigma_5 \tau_{(6)}} \bigg)
\Delta\theta^{\#}(\Li^{\fu}_{3}) \Phi_{e_{0}}^{\tau_{(4)}}\Phi_{e_{0}}^{\tau_{(5)}}\Phi_{e_{0}}^{\tau_{(6)}} \\
 +  
\sum_{\tau_{(2)} ,\ldots, \tau_{(6)}} \frac{1}{f_{\sigma_{3}}f_{\sigma_5}} \bigg[ \bigg( f_{\tau_{p} \tau_{(2)} \tau_{(3)}} f_{\sigma_{3} \tau_{(4)} \tau_{(5)}} f_{\sigma_5 \tau_{(6)}} - f_{\tau_{p}\tau_{(2)}\tau_{(3)}\tau_{(4)} \tau_{(5)}} f_{\sigma_{3}} f_{\sigma_5 \tau_{(6)}} \bigg. \bigg. \\
 - f_{\tau_{p} \tau_{(2)} \tau_{(3)}} f_{\sigma_3 \tau_{(4)}\tau_{(5)}\tau_{(6)}}f_{\sigma_5} + f_{\tau_{p} \tau_{(2)}\ldots \tau_{(6)}} f_{\sigma_3} f_{\sigma_5} \bigg)\Delta\Phi_{e_{1}}^{\tau_{p}} \\
 \bigg. \bigg. + 
\bigg( f_{\tau_{q} \tau_{(2)} \tau_{(3)}} f_{\sigma_{3} \tau_{(4)} \tau_{(5)}} f_{\sigma_5 \tau_{(6)}} - f_{\tau_{q}\tau_{(2)}\tau_{(3)}\tau_{(4)} \tau_{(5)}} f_{\sigma_{3}} f_{\sigma_5 \tau_{(6)}} \bigg. \bigg. \\
 - f_{\tau_{q} \tau_{(2)} \tau_{(3)}} f_{\sigma_3 \tau_{(4)}\tau_{(5)}\tau_{(6)}}f_{\sigma_5} + f_{\tau_{q} \tau_{(2)}\ldots \tau_{(6)}} f_{\sigma_3} f_{\sigma_5} \bigg)\Delta\Phi_{e_{1}}^{\tau_{q}} \bigg]
 \Phi_{e_{0}}^{\tau_{(2)}} \Phi_{e_{0}}^{\tau_{(3)}}
 \Phi_{e_{0}}^{\tau_{(4)}} 
 \Phi_{e_{0}}^{\tau_{(5)}}\Phi_{e_0}^{\tau_{(6)}}.
\end{multline}
Use Lemmas \ref{lem:1p} and \ref{lem:1q} in (\ref{eq:theta Li4 bis bis bis}) to eliminate $\Delta \Phi_{e_{1}}^{\tau_{p}}$ and $\Delta \Phi_{e_{1}}^{\tau_{q}}$, and Lemma \ref{eq:detM} to eliminate $\Delta$:
\begin{multline}\label{eq:one before last P6}
\bigg( \frac{1}{2} f_{[\tau_q, \tau_p]}\theta^{\#}(\log^{\fu}) \bigg)\theta^{\#}(\Li^{\fu}_{6}) \\ =  \frac{f_{\sigma_5 \tau_p}}{f_{\sigma_5}} \bigg( \frac{1}{2} f_{[\tau_q, \tau_p]}\theta^{\#}(\log^{\fu}) \bigg)
\theta^{\#}(\Li^{\fu}_{5}) \Phi_{e_0}^{\tau_p} + \frac{f_{\sigma_5 \tau_q}}{f_{\sigma_5}} \bigg( \frac{1}{2} f_{[\tau_q, \tau_p]}\theta^{\#}(\log^{\fu}) \bigg)
\theta^{\#}(\Li^{\fu}_{5}) \Phi_{e_0}^{\tau_q} \\
+ \sum_{\tau_{(4)}, \tau_{(5)}, \tau_{(6)}}\frac{1}{f_{\sigma_{3}} f_{\sigma_5}} \bigg( f_{\sigma_3 \tau_{(4)}\tau_{(5)}\tau_{(6)}} f_{\sigma_5} - f_{\sigma_{3} \tau_{(4)} \tau_{(5)}} f_{\sigma_5 \tau_{(6)}} \bigg)
\bigg( \frac{1}{2} f_{[\tau_q, \tau_p]}\theta^{\#}(\log^{\fu}) \bigg)\theta^{\#}(\Li^{\fu}_{3}) \Phi_{e_{0}}^{\tau_{(4)}}\Phi_{e_{0}}^{\tau_{(5)}}\Phi_{e_{0}}^{\tau_{(6)}} \\
 +  
\sum_{\tau_{(2)} ,\ldots, \tau_{(6)}} \frac{1}{f_{\sigma_{3}}f_{\sigma_5}} \bigg[ \bigg( f_{\tau_{p} \tau_{(2)} \tau_{(3)}} f_{\sigma_{3} \tau_{(4)} \tau_{(5)}} f_{\sigma_5 \tau_{(6)}} - f_{\tau_{p}\tau_{(2)}\tau_{(3)}\tau_{(4)} \tau_{(5)}} f_{\sigma_{3}} f_{\sigma_5 \tau_{(6)}} \bigg. \bigg. \\
 - f_{\tau_{p} \tau_{(2)} \tau_{(3)}} f_{\sigma_3 \tau_{(4)}\tau_{(5)}\tau_{(6)}}f_{\sigma_5} + f_{\tau_{p} \tau_{(2)}\ldots \tau_{(6)}} f_{\sigma_3} f_{\sigma_5} \bigg)\bigg( \frac{1}{2}f_{[\tau_q,\tau_p]}\theta^{\#}(\Li_1^{\fu})\Phi_{e_0}^{\tau_p}-f_{\tau_q} \theta^{\#}(F^{\vert 1\vert}_{2,2}) \bigg) \\
 \bigg. \bigg. + \frac{1}{f_{\tau_q}}
\bigg( f_{\tau_{q} \tau_{(2)} \tau_{(3)}} f_{\sigma_{3} \tau_{(4)} \tau_{(5)}} f_{\sigma_5 \tau_{(6)}} - f_{\tau_{q}\tau_{(2)}\tau_{(3)}\tau_{(4)} \tau_{(5)}} f_{\sigma_{3}} f_{\sigma_5 \tau_{(6)}} 
 - f_{\tau_{q} \tau_{(2)} \tau_{(3)}} f_{\sigma_3 \tau_{(4)}\tau_{(5)}\tau_{(6)}}f_{\sigma_5} \bigg. \bigg. \\
 + f_{\tau_{q} \tau_{(2)}\ldots \tau_{(6)}} f_{\sigma_3} f_{\sigma_5} \bigg)\bigg( (f_{\tau_p^2\tau_q}-f_{\tau_q\tau_p^2})\theta^{\#}(\Li_1^{\fu})\Phi_{e_0}^{\tau_p}
    +\big( f_{\tau_p}f_{\tau_q}\theta^{\#}(\Li_2^{\fu})-f_{\tau_p\tau_q}\theta^{\#}(\log^{\fu}\Li_1^{\fu})\big) \bigg) \bigg] \\
 \cdot \Phi_{e_{0}}^{\tau_{(2)}} \Phi_{e_{0}}^{\tau_{(3)}}
 \Phi_{e_{0}}^{\tau_{(4)}} 
 \Phi_{e_{0}}^{\tau_{(5)}}\Phi_{e_0}^{\tau_{(6)}}.
\end{multline}
Inserting equations (\ref{eq:theta log for two primes'}) and (\ref{eq:function}) in equation (\ref{eq:one before last P6}) to eliminate $\Phi_{e_{0}}^{\tau_{q}}$, and clearing denominators by multiplying by $f_{\tau_q}^6 f_{\sigma_3} f_{\sigma_5}$, we obtain an equation of the form $P_{6}(\Phi_{e_{0}}^{\tau_{p}})=0$ where $P_{6}(X) = \theta^{\sharp}(\nu_{6}(X))$ with
\begin{multline}\label{eq:nu6}
\nu_6(X)= \\
-\frac{1}{2} f_{\tau_q}^6 f_{\sigma_3} f_{\sigma_5} f_{[\tau_q, \tau_p]}\log^{\fu} \Li^{\fu}_{6}
+ \frac{1}{2} f_{\tau_q}^5 f_{\sigma_3} f_{\sigma_5 \tau_q}  f_{[\tau_q, \tau_p]}(\log^{\fu})^2 
\Li^{\fu}_{5}
 \\ + \frac{1}{2} f_{\tau_q}^5 f_{\sigma_3} f_{[\tau_q, \tau_p]} \big(f_{\sigma_5 \tau_p} f_{\tau_q}  - f_{\sigma_5 \tau_q} f_{\tau_p} \big)\log^{\fu} 
\Li^{\fu}_{5} X \\
+ \sum_{j=0}^3\sum_{\substack{J\subset \{ 4,5,6 \} \\ \vert J\vert=j }} \frac{1}{2} f_{\tau_q}^{3+j}f_{[\tau_q, \tau_p]} \bigg( f_{\sigma_3 \tau_{(4)}^J\tau_{(5)}^J\tau_{(6)}^J} f_{\sigma_5} - f_{\sigma_{3} \tau_{(4)}^J \tau_{(5)}^J} f_{\sigma_5 \tau_{(6)}^J} \bigg)
\log^{\fu}\Li^{\fu}_{3}X^j \big( \log^{\fu}-f_{\tau_p}X \big)^{3-j} \\
 +  
\sum_{i=0}^5 \sum_{\substack{I\subset \{ 2, \ldots, 6 \} \\ \vert I \vert=i }} f_{\tau_q}^{i+1} \bigg[ \bigg( f_{\tau_{p} \tau_{(2)}^I \tau_{(3)}^I} f_{\sigma_{3} \tau_{(4)}^I \tau_{(5)}}^I f_{\sigma_5 \tau_{(6)}^I} - f_{\tau_{p}\tau_{(2)}^I\tau_{(3)}^I\tau_{(4)}^I \tau_{(5)}^I} f_{\sigma_{3}} f_{\sigma_5 \tau_{(6)}^I} \bigg. \bigg. \\
 - f_{\tau_{p} \tau_{(2)}^I \tau_{(3)}^I} f_{\sigma_3 \tau_{(4)}^I\tau_{(5)}^I\tau_{(6)}^I}f_{\sigma_5} + f_{\tau_{p} \tau_{(2)}^I\ldots \tau_{(6)}^I} f_{\sigma_3} f_{\sigma_5} \bigg)\bigg(  \frac{1}{2} f_{[\tau_q, \tau_p]}\Li^{\fu}_1 X- f_{\tau_q}F^{\vert 1\vert}_{2,2} \bigg) \\
 \bigg. \bigg. + \frac{1}{f_{\tau_q}} 
\bigg( f_{\tau_{q} \tau_{(2)}^I \tau_{(3)}^I} f_{\sigma_{3} \tau_{(4)}^I \tau_{(5)}^I} f_{\sigma_5 \tau_{(6)}^I} - f_{\tau_{q}\tau_{(2)}^I\tau_{(3)}^I\tau_{(4)}^I \tau_{(5)}^I} f_{\sigma_{3}} f_{\sigma_5 \tau_{(6)}^I} 
 - f_{\tau_{q} \tau_{(2)}^I \tau_{(3)}^I} f_{\sigma_3 \tau_{(4)}^I\tau_{(5)}^I\tau_{(6)}^I}f_{\sigma_5} \bigg. \bigg. \\
 + f_{\tau_{q} \tau_{(2)}^I\ldots \tau_{(6)}^I} f_{\sigma_3} f_{\sigma_5} \bigg)\bigg( (f_{\tau_p^2\tau_q}-f_{\tau_q\tau_p^2})\Li_1^{\fu} X + \big( f_{\tau_p}f_{\tau_q}\Li_2^{\fu}-f_{\tau_p\tau_q}\log^{\fu}\Li_1^{\fu}\big) \bigg) \bigg]
 X^i \big( \log^{\fu}-f_{\tau_p}X \big)^{5-i}.
\end{multline}

\begin{lemma}\label{lem:4cst}
    The polynomial $\nu_6(X)$ has degree $4$ and is a multiple of $\log^{\fu}$. 
\end{lemma}

\begin{proof}
    The statement about the degreelwas verified using SageMath \cite{sagemath} (see our  \href{https://github.com/Ariel-Z-Weiss/polylog-chabauty-kim}{GitHub repository} for the code). It is then clear from $\eqref{eq:nu6}$ that $\nu_6(X)$ is a multiple of $\log^{\fu}$. 
\end{proof}

\subsection{The Chabauty--Kim function $F^{\vert 2\vert}_{6,18}$}

Writing
\begin{align*}
    \nu_{4}(X) & = a_3X^2 +a_4X + a_{5} \in \oh(U_S)[\log^{\fu}, \Li_1^{\fu}, \ldots, \Li_4^{\fu}] \\
    \nu_{6}(X) & = b_3X^4+b_4X^3+b_5X^2+b_6X + b_{7} \in \oh(U_S)[\log^{\fu}, \Li_1^{\fu}, \ldots, \Li_6^{\fu}],
\end{align*}
one way to express their resultant is as the determinant of the following $6\times 6$ matrix with coefficients in $\oh(U_S)$:
\begin{equation}\label{resmatrixassu}
\mathrm{Res}(\nu_4(X), \nu_6(X))=
\left\vert
\begin{matrix}
a_3 & 0 & 0 & 0 & b_3 & 0 \\
a_4 & a_3 & 0 & 0 & b_4 & b_3 \\
a_5 & a_4 & a_3 & 0 & b_5 & b_4 \\
0 & a_5 & a_4 & a_3 & b_6 & b_5 \\
0 & 0 & a_5 & a_4 & b_7 & b_6 \\
0 & 0 & 0 & a_5 & 0 & b_7 \\
\end{matrix}
\right\vert.
\end{equation}

The indexing of the coefficients of the polynomials is such that $a_i, b_j \in \oh(U_S)[\log^{\fu}, \Li_1^{\fu}, \ldots, \Li_6^{\fu}]$ are homogeneous of degrees $i$ and $j$, respectively. This is clear from equations \eqref{eq:nu4} and \eqref{eq:nu6}. Note that $\mathrm{Res}(\nu_4(X), \nu_6(X))$ is homogeneous of degree $\deg((a_3)^4(b_7)^2)=12+14=26$. But all the coefficients $a_i$ and $b_j$ are multiples of $\log^{\fu}$ by Lemmas \ref{lem:4cst2} and \ref{lem:4cst}. Writing $a_i=a_i'\log^{\fu}$ and $b_j=b_j'\log^{\fu}$, we observe that 
\begin{equation}\label{resmatrixassu2}
\mathrm{Res}(\nu_4(X), \nu_6(X))=(\log^{\fu})^6
\left\vert
\begin{matrix}
a'_3 & 0 & 0 & 0 & b'_3 & 0 \\
a'_4 & a'_3 & 0 & 0 & b'_4 & b'_3 \\
a'_5 & a'_4 & a'_3 & 0 & b'_5 & b'_4 \\
0 & a'_5 & a'_4 & a'_3 & b'_6 & b'_5 \\
0 & 0 & a'_5 & a'_4 & b'_7 & b'_6 \\
0 & 0 & 0 & a'_5 & 0 & b'_7 \\
\end{matrix}
\right\vert,
\end{equation}
where the latter resultant is of degree $20$. 

\begin{lemma}\label{prop:34}
    The coefficients $a_3'$ and $b_3'$ are multiples of $F^{\vert 1\vert}_{2,2}$.
\end{lemma}

\begin{proof}
    We have verified this in SageMath \cite{sagemath}. See our  \href{https://github.com/Ariel-Z-Weiss/polylog-chabauty-kim}{GitHub repository} for the code.
\end{proof}

\begin{definition}\label{def:F618}
    Define $F^{\vert 2\vert}_{6,18}\in \oh(U_S)[\log^{\fu}, \Li_1^{\fu}, \ldots, \Li_6^{\fu}]$ by 
    \[
    \mathrm{Res}(\nu_4(X), \nu_6(X))=(\log^{\fu})^6 F^{\vert 1\vert}_{2,2} F^{\vert 2\vert}_{6,18}.
    \]
\end{definition}

\begin{theorem}\label{prop:F5} 
The polynomial $F^{\vert 2\vert}_{6,18}$ is a non-trivial polylogarithmic motivic Chabauty--Kim function of depth $6$ and degree $18$.
\end{theorem}

\begin{proof}
    Since $\mathrm{Res}(\nu_4(X), \nu_6(X))$ has degree $26$, it is clear that the degree of $F^{\vert 2\vert}_{6,18}$ is $18$.
    In view of Definition \ref{def:F618} and Proposition \ref{eq: resultant vanishes}, the only part of the statement that requires justification is the non-triviality. 
    It suffices to prove the non-triviality of $\mathrm{Res}(\nu_4(X), \nu_6(X))$. We observe from \eqref{eq:nu6} that the constant coefficients of $\nu_6(X)$ is of the form:
    \[
    b_7=-\frac{1}{2}f_{\tau_q}^6 f_{\sigma_3} f_{\sigma_5}f_{[\tau_q,\tau_p]} \log^{\fu} \Li^{\fu}_6 + \beta, \qquad \text{ for some } \beta \in \oh(U_S)[\log^{\fu}, \Li_1^{\fu}, \ldots, \Li_5^{\fu}]_7.
    \]
    In particular, we see that $b_7\neq 0$.
    Moreover, it is clear from \eqref{eq:nu6} that $b_j\in \oh(U_S)[\log^{\fu}, \Li_1^{\fu}, \ldots, \Li_5^{\fu}]$ for $3\leq j\leq 6$. In other words, only the constant term of $\nu_6(X)$ involves the variable $\Li^{\fu}_6$. On the other hand, the leading coefficient of $\nu_4(X)$ is the explicit non-zero quantity $a_3$ that can be read off \eqref{eq:nu4}. When writing out the determinant \eqref{resmatrixassu} as a cofactor expansion, we see that $\mathrm{Res}(\nu_4(X), \nu_6(X))\in \oh(U_S)[\log^{\fu}, \Li_1^{\fu}, \ldots, \Li_6^{\fu}]$ contains the following term arising from the diagonal product $(a_3)^4(b_7)^2$:  
    \[
    \frac{1}{4}a_3^4 f_{\tau_q}^{12} (f_{\sigma_3} f_{\sigma_5}f_{[\tau_q,\tau_p]})^2 (\log^{\fu})^2(\Li_6^{\fu})^2.
    \]
    By the above remarks, the diagonal contribution to the determinant is the only term which is a multiple of $(\Li_6^{\fu})^2$. Since $a_3\neq 0$ by Lemma \ref{lem:4cst2}, it follows that $\mathrm{Res}(\nu_4(X), \nu_6(X))\neq 0$.
\end{proof}

\subsection{Chabauty--Kim functions of higher degree}

We have constructed polynomials $P_{4}(X)$ and $P_{6}(X)$, as well as $\nu_{4}(X)$ and $\nu_{6}(X)$, using as starting point the equations \eqref{eq:theta Lin for two primes pair} for $\theta^{\#}(\Li^{\fu}_{4})$ and $\theta^{\#}(\Li^{\fu}_{6})$. More generally, a similar process of elimination as in \Cref{s:elim} would give rise, for any $d\geq 2$, to polynomials $P_{2d}(X) \in \theta^{\#}(\oh(U_S)[\log^{\fu}, \Li_1^{\fu}, \ldots, \Li_{2d}^{\fu}])$ and $\nu_{2d}(X) \in \oh(U_S)[\log^{\fu}, \Li_1^{\fu}, \ldots, \Li_{2d}^{\fu}]$ such that $P_{2d}(\Phi_{e_0}^{\tau_p})=0$ and $P_{2d}(X)=\theta^{\#}(\nu_{2d}(X))$. Given $d_1 > d_2 \geq 2$, we then obtain polylogarithmic motivic Chabauty--Kim functions
$$ \mathrm{Res}(\nu_{2d_1}(X),\nu_{2d_2}(X))\in \cI^{\PL}_{S, 2d_1}.$$
A similar proof as in Theorem \ref{prop:F5} then guarantees that $$\mathrm{Res}(\nu_{2d_1}(X),\nu_{2d_2}(X))\neq 0.$$
The resultant method for $\vert S\vert=2$ thus provides
an infinite family of Chabauty--Kim functions.
In practice, these other functions will be of high degree and the complexity of the arithmetic step will increase many-fold. \footnote{We expect the degree of $\nu_{2d}(X)$ to be $2d-2$. In fact, we checked this for $d=2, 3, 4, 5$ in SageMath \cite{sagemath}. We also verified that the leading coefficient of $\nu_{2d}(X)$ is a multiple of $F^{\vert 1\vert}_{2,2}$ for $d=2,3,4,5$.} 

\section{Towards the multiple primes case}\label{s:multiple}

We briefly outline a strategy using resultants to tackle the geometric step for a set $S$ with no restriction on its size. 

Let $s=\vert S\vert$. We then have $\Sigma_{-1} = \{\tau_{p_{1}},\ldots,\tau_{p_{s}}\}$. The map $\theta^{\#}$ is given by the usual formulas in Definition \ref{theta}.
The system of equations is therefore similar to the one in \Cref{s:equations}, except that the sums that are indexed by $\Sigma_{-1}$ now have $s$ terms. In particular, the system is still linear with respect to $\Phi^{\sigma_{2n-1}}_{e_{1}e_{0}^{2n-2}}$ and the $\Phi_{e_{1}}^{\tau}$'s, but now the number of variables $\Phi_{e_{0}}^{\tau}$ is $s$, as is the number of variables $\Phi_{e_{1}}^{\tau}$.

A generalisation of the elimination process described in \Cref{s:elim} goes as follows. Take as starting point the equation for $\theta^{\#}(\Li^{\fu}_{2d})$ and proceed to eliminate the variables $\Phi^{\sigma_{2n-1}}_{e_{1}e_{0}^{2n-2}}$ with $n\leq d$ by induction on $n$, starting with $n=d$. Then proceed to invert a linear system to eliminate the $\Phi_{e_{1}}^{\tau}$'s, provided that it is invertible. It is unclear whether the determinant of the corresponding matrix will have as nice a shape as the one in Lemma \ref{eq:detM}.  
After inverting the system, only the $s$ variables $\Phi_{e_{0}}^{\tau}$ remain. In order to eliminate these, proceed by induction on $n$ as follows. 

Assume that we have polynomials $P_{1},\ldots,P_{s} \in \oh(U_S)[X_{1},\ldots,X_{s}]$ with a common root 
$(x_{1},\ldots,x_{s})$. Then for each $2 \leq i \leq n$, we have 
\[
\mathrm{Res}_{X_{s}}(P_{1}(x_{1},\ldots,x_{s-1},X_{s}), P_{i}(x_{1},\ldots,x_{s-1},X_{s})) =0.
\]
In particular, the polynomials 
\[
Q_{i}(X_{1},\ldots,X_{s-1}) = \mathrm{Res}_{X_{s}}(P_{1}(X_{1},\ldots,X_{s-1},X_{s}), P_{i}(X_{1},\ldots,X_{s-1},X_{s}))
\]
have the common root $(x_{1},\ldots,x_{s-1})$. Proceed by induction from here, by iteratively taking resultants of resultants. 
The outcome of this procedure is an element in $\ker(\theta^{\#})$. Verifying the non-triviality of such a function seems non-trivial when $s>2$ (in Theorem \ref{prop:F5} we proved it in the case $s=2$). Moreover, in the case $s=2$ we noticed that $\Res(\nu_4(X), \nu_6(X))$ is a multiple of $(\log^{\fu})^6 F^{\vert 1\vert}_{2,2}$, an observation that allowed us to extract a polylogarithmic Chabauty--Kim function of lower degree. We do not know if a similar phenomenon happens when $s>2$.

\begin{remark}
Alternatively, one could use the Macaulay resultant, which applies directly to a family of any number of polynomials (and not just two polynomials). This remark is also relevant in the $s=2$ case. However, since all the steps of the elimination in the case $s=2$ are linear except for the elimination step of $\Phi_{e_0}^{\tau_p}$, the method presented using classical resultants seemed the most natural choice.
\end{remark}

%%%%%%%%%%%%%%%%%%%%%%%%%%%%%%%%%%%%%%%%%%%%%%%%%%%%%%%%%%%%

\end{document}